\documentclass[11pt]{article}
\usepackage{amsmath}
\usepackage{amsfonts}
\usepackage{amssymb}
\usepackage{amsthm}
\usepackage{graphicx}
\usepackage{psfrag}
\usepackage{cite}
\usepackage{url}
\usepackage[usenames]{color} 

\newtheorem{theorem}{Theorem}[section]
\newtheorem{proposition}[theorem]{Proposition}
\newtheorem{definition}[theorem]{Definition}
\newtheorem{lemma}[theorem]{Lemma}
\newtheorem{corollary}[theorem]{Corollary}
\newtheorem{remark}[theorem]{Remark}

\theoremstyle{remark}

\newcommand{\fraku}{\mathfrak{U}}

\newcommand{\spn}{\mathop{\mathrm{span}}}

\newcommand{\nats}{\mathbb{N}}
\newcommand{\reals}{\mathbb{R}} 
\newcommand{\RR}{\mathbb{R}}
 
\newcommand{\CC}{\mathbb{C}} 
\newcommand{\LL}{\mathbb{L}} 
\newcommand{\M}{\mathbb{M}}
\newcommand{\calo}{\mathcal{O}}

\newcommand{\sph}{\mathbb{S}} 

\newcommand{\dif}{\mathrm{d}}

\renewcommand{\d}{\mathrm{dist}}

\newcommand{\J}{\mathcal{J}} 

\renewcommand{\L}{\mathcal{L}}

\newcommand{\rproj}{\mathbb{RP}}
\newcommand{\vol}{\mathrm{vol}}
\newcommand{\bfone}{\mathbf{1}}

\def\calg{{\mathcal G}}
\def\caln{{\mathcal N}}

\numberwithin{equation}{section}
\title{Kernel Based Quadrature  on Spheres and Other Homogeneous Spaces}
 \author{E.~Fuselier\thanks{ Department of Mathematics, High Point University,
 High Point, NC 27262, USA.}, 
 T.~Hangelbroek\thanks{ Department of Mathematics, University of Hawaii, 
 Honolulu, HI 96822, USA.  Research supported by  by grant DMS-1232409 from the National
    Science Foundation.}, 
F. J.~Narcowich\thanks{ Department of Mathematics, Texas A\&M
    University, College Station, TX 77843, USA. Research
    supported by grant DMS-1211566 from the National
    Science Foundation.}, 
J. D.~Ward\thanks{ Department of Mathematics, Texas A\&M University,
    College Station, TX 77843, USA. Research supported by
    grant DMS-1211566 from the National Science
    Foundation.},
 G. B~Wright\thanks{Department of Mathematics, Boise State University, Boise, ID 83725, USA.
    Research supported by grants DMS-0934581and DMS-1160379 from the National Science
    Foundation.} 
    }

\begin{document}
\maketitle

\begin{abstract} Quadrature formulas for spheres, the rotation group, and other
compact, homogeneous manifolds are important in a number of applications and have been the subject
of recent research. The main purpose of this paper is to study
coordinate independent quadrature (or cubature) formulas associated with certain classes of
positive definite and conditionally positive definite kernels that are
invariant under the group action of the homogeneous manifold.  In particular, we
show that these formulas are accurate -- optimally so in many cases --, and stable under an increasing number of nodes and in the presence of noise, provided the set $X$
of quadrature nodes is quasi-uniform.  The stability results are new in all cases.  In addition, we may use these quadrature formulas to obtain similar formulas for manifolds diffeomorphic to $\sph^n$, oblate spheroids for instance. The weights are obtained by solving a single linear system.  For $\sph^2$, and the restricted thin plate spline kernel $r^2\log r$, these weights can be computed for two-thirds of a million nodes, using a preconditioned iterative technique introduced by us.

\end{abstract}

\section{Introduction}
\label{intro}

Quadrature formulas for spheres, the rotation group, and other
compact, homogeneous manifolds are important in many applications and have been the subject
of recent research \cite{Mhaskar-etal-01-1, Narcowich-etal-06-1,
  graef-kunis-potts-2009, hesse-et-al-2010, graef-potts-2009,
  Pesenson-Geller-2011}. The main purpose of this paper is to study
quadrature (or cubature) formulas associated with certain classes of
positive definite and conditionally positive definite kernels that are
invariant under the group action of the homogeneous manifold, and to
show that these formulas are accurate and stable, provided the set $X$
of quadrature nodes is quasi-uniform.  The invariance of the
kernels is a key ingredient in giving simple, easy to construct linear
systems that determine the weights. The weights themselves have size
comparable to $1/N$, where $N$ is the number nodes in $X$.  For $\sph^2$, and the restricted thin plate spline
kernel $r^2\log r$, these weights can be computed for very large values of $N$ using the preconditioned iterative technique
described in~\cite{FHNWW2012}.

The quadrature formulas  developed here are for the general setting of a compact,
homogeneous, $n$-dimensional manifold $\M$ that is
equipped with a group invariant Riemannian metric $g_{ij}$ and its
associated invariant measure $d\mu(x)=\sqrt{\det(g_{ij}(x))}dx$.  From the point of view of current applications,
the two most important of these manifolds are $\sph^2$ and
$SO(3)$. However, other homogeneous spaces, such as Stiefel and
Grassmann manifolds, are arising in applications \cite{
  absil_etal_2008, edelman_1999}, and we expect that in the future our
results will be applied to them. The
type of quadrature formula we will be concerned with here has the form
\begin{equation}
\label{quadrature_formula_general}
\int_\M f(x)d\mu(x) \doteq \sum_{\xi \in X}c_\xi f(\xi) =:Q(f), \ f\in C(\M),
\end{equation}
where the set $X\subset \M$ of centers/nodes is finite. The weights
$\{c_\xi\}_{\xi\in X}$ are chosen so that quadrature operator $Q$
integrates exactly a given finite dimensional space of continuous
functions, $V$.

In the case of $\sph^n$ and $SO(3)$, popular choices for $V$ are
spaces of spherical harmonics \cite{Mhaskar-etal-01-1,
  Narcowich-etal-06-1, graef-kunis-potts-2009, hesse-et-al-2010}, and
Wigner D-functions \cite{graef-potts-2009}. Work also has been done on
compact two-point homogeneous manifolds \cite{Brown-Dai-05-1} and on
general compact homogeneous manifolds \cite{Pesenson-Geller-2011},
with $V$ being chosen to be a set of eigenfunctions of a manifold's
Laplace-Beltrami operator. We will use the term \emph{polynomial
  quadrature} for methods that integrate such spaces exactly; this is
because, for spheres and a few other spaces, $V$ consists of
restrictions of harmonic polynomials.

The quadrature methods developed here are \emph{kernel}
methods; they use a space $V$ consisting of linear combinations of
kernels. In \cite{sommariva_womer2005}, Sommariva and Womersley  used spaces of rotationally
invariant radial basis functions (RBFs) and spherical basis functions (SBFs) to derive a  linear system of equations for the weights $c_\xi$. However, neither accuracy, control over the size of the weights, nor stability was  addressed in \cite{sommariva_womer2005}. In this paper these and other issues are dealt with employing recent results developed by us in \cite{HNW2010, HNW_3_2011, FHNWW2012}.

Accuracy is measured in terms of the mesh norm $h$, which is defined
in section~\ref{kernel_interp_quad}. All the kernels we deal with are
associated with a Sobolev space $W_2^m$, for some $m$ in $\nats$.
Previous error estimates for quadrature with positive definite kernels
on $\sph^2$ were given in \cite{hesse-et-al-2010}. For the general
case of a homogeneous manifold, the error we obtained is
$\calo (h^m)$, for functions in $W_2^m$. However, on $\sph^n$ or
$SO(3)$, we get even better rates, in fact optimal: if a function is in $C^{2m}$,
then the order is $\calo(h^{2m})$. For example, the thin-plate
spline kernel restricted to $\sph^2$ has $m=2$, so, for a
function in $C^4(\sph^2)$, the error would be $\calo (h^4)$.

In the course of studying accuracy for $Q$ in various spaces, we obtain new error estimates for interpolation/reproduction problems on two-point homogeneous manifolds. First,  for a general such manifold,  the estimates  hold for $f\in W_2^m(\M)$, which is the ``native space'' of the kernel.  Second, for spheres (new when reproduction is required) and real projective spaces, the estimates apply to $f\in W_2^\mu$, with $n/2 <\mu \le m$. Thus the estimates allow an ``escape'' from the native space, in the sense that they are for functions \emph{not} smooth enough to be in that space.

These quadrature formulas are stable both under an increase in the
number of points and in the presence of noise. If the number of points
is increased, then the norm of the quadrature operator remains uniformly
bounded, as long as the level of quasi-uniformity is maintained. Thus
there is no oscillatory ``Runge phenomenon." To examine the effect of noise, we
assume the measured function values differ from the actual ones by
independent, identically distributed, zero mean random
variables. Under these conditions, the standard deviation of the
quadrature formula decreases as $\calo (N^{-1/2})$.

To illustrate how the method works, consider a positive definite SBF kernel $\phi(x\cdot y)$, $x,y\in \sph^n$, and let $V = \spn \{\phi(x\cdot \xi)\}_{\xi\in X}$. We want to integrate $s \in V$ exactly; that is, we require
$\sum_{\xi\in X}c_\xi s(\xi) = \int_{\sph^n} s(x)d\mu(x)$. Doing so
results in a systems of linear equations for the weights.  If $A$ is
the interpolation matrix with entries $A_{\xi,\eta}=\phi(\xi\cdot
\eta)$, $\xi,\eta\in X$, then the vector of weights $c$ satisfies
$(Ac)_\xi=\int_{\sph^n} \phi(\xi \cdot x)d\mu(x)$. At this point, we
encounter an apparent difficulty. We have to compute \emph{every} one of the
integrals $\int_{\sph^n} \phi(\xi \cdot x)d\mu(x)$ for \emph{every}
$\xi\in X$.

The rotational invariance of the SBF allows us to overcome this
difficulty. Because of rotational invariance, all of the integrals are
independent of $\xi$, and thus have the same value $J_0$; the system
then becomes $(Ac)_\xi=J_0$. The constant $J_0$ only needs to be
computed \emph{once} for a given kernel. (For many SBFs/RBFs, values
of $J_0$ are known; see \cite{sommariva_womer2005}.) The same is true
for group invariant kernels on $\M$, as we will see in
section~\ref{quadrature} below. The point to be emphasized here is
that group invariance of the kernels allows us to deal with quadrature
as an interpolation problem. Without it, the problem requires
computing many integrals and in fact becomes prohibitively expensive, computationally.

Numerical tests of the quadrature formulas were carried out for $\M=\sph^2$, in connection with the SBF  kernel $\Phi(x\cdot y)=(1-x\cdot y)\log(1-x\cdot y)$, $x,y\in \sph^2$. This kernel is the thin-plate spline $r^2\log r$ restricted to the sphere, and it corresponds to one of the Sobolev spaces mentioned above, namely,  $W_2^2(\sph^2)$. For these tests, the sets of quasi-uniform  nodes were generated via three different, commonly used methods: icosahedral,  Fibonacci (or \emph{phyllotaxis}), and quasi-minimum energy. The number of nodes employed varied over a substantial range, from a few thousand to two-thirds of a million. Weights corresponding to these nodes were computed using a pre-conditioning method developed in \cite{FHNWW2012}. The tests themselves focused on the accuracy and stability of the method, both in terms of increasing the number of nodes and adding noise. The tests, which are discussed in section~\ref{numerics}, gave excellent results, in agreement with the theory. 

There are situations where the manifold $\M$ involved is \emph{not} a homogeneous space, but  quadrature formulas can still be obtained. If $\M$ is diffeomorphic  to a homogeneous space, then it is possible to obtain quadrature weights for $\M$ from the ones for the corresponding homogeneous space. In section~\ref{manif_diffeo}, we will show how this can be done for $\sph^n$. We will then apply this to the specific case where $\M$ is an oblate spheroid (e.g., earth with flattening accounted for), which is of course diffeomorphic to $\sph^2$.

The paper is organized this way. Section~\ref{kernel_interp_quad} begins with a brief discussion of  positive definite/conditionally positive definite kernels, notation, and, in section~\ref{interpolation}, interpolation. Section~\ref{quadrature} contains a derivation and discussion of kernel quadrature formulas, with special emphasis on the role played by group invariance of the kernel employed in the formula. In section~\ref{accuracy_stability}, the questions of accuracy and stability mentioned in the introduction are taken up. The results obtained there are aimed at invariant kernels,  such as Sobolev and  polyharmonic kernels discussed in sections \ref{sobolev_kernels} and \ref{polyharmonic_kernels}. 

Sobolev kernels on a compact manifold are positive definite reproducing kernels for the Sobolev space $W_2^m$, $m>n/2$. In section~\ref{sobolev_kernels}, we study these in terms of their invariance, interpolation errors, and properties of their Lagrange functions and Lebesgue constants. Finally, in section~\ref{kappa_m_quadrature} we look at their use in  quadrature formulas. 

Section~\ref{polyharmonic_kernels} is devoted to a very important class of kernels on a compact, two-point homogeneous manifold: the polyharmonic kernels. These kernels, which may be either positive definite or conditionally positive definite, are Green's functions for operators that are polynomials in the the Laplace-Beltrami operator. On spheres, they include restricted surface splines, and on $SO(3)$ similar kernels. All of these are given in terms of simple, explicit formulas.  The whole section is a self-contained discussion of these kernels, culminating in their application to quadrature formulas.

The results from various numerical tests  that we conducted are discussed in detail in section~\ref{numerics}. Finally, in section~\ref{manif_diffeo} we discuss ways of using quadrature weights for  a compact, homogeneous manifold $\M$ to obtain invariant, coordinate independent weights for manifolds diffeomorphic to $\M$.

\section{Interpolation and Quadrature via Kernels}
\label{kernel_interp_quad}

The spaces that we will work with will be homogeneous manifolds,
ultimately. However, for interpolation we need very little in the way
of structure. In fact, we could take  our underlying space $\M$ to
be a metric space.

The set $X\subset \M$ will be assumed finite. Its \emph{mesh norm} (or
\emph{fill distance}) $h:=\sup_{x\in \M} \d(x,X)$ measures the density
of $X$ in $\M$, while the \emph{separation radius} $q:=\frac12
\inf_{\substack{\xi,\zeta\in X\\ \xi\ne \zeta}} \d(\xi,\zeta)$
determines the spacing of $X$. The \emph{mesh ratio} $\rho:=h/q$
measures the uniformity of the distribution of $X$ in $\M$.

We say that a continuous kernel $\kappa: \M\times\M\to \RR$ is
(strictly) \emph{positive definite} on $\M$ if, for every finite
subset $X\subset \M$, the matrix $A$ with entries
$A_{\xi,\eta}:=\kappa(\xi,\eta)$, $\xi,\eta\in X$, is positive
definite. \emph{Conditionally (strictly) positive definite} kernels
are defined with respect to a  finite dimensional space
$\Pi:=\spn\{\psi_k:\M\to \RR\}_{k=1}^m$, where the $\psi_k$'s are
linearly independent, continuous functions on $\M$. In addition,
given a finite set of centers $X\subset \M$, where we let $N:=\#
X$ be the cardinality of $X$, we say that $\Pi$ is \emph{unisolvent}
on $X$ if the only function $\psi\in \Pi$ for which $\psi|_X=0$ is
$y\equiv 0$. This means that $\{\psi_k|_X\}_{k=1}^m$ is a linearly
independent set in $\RR^N$. Given that $\Pi$ is unisolvent on $X$, we
say that the kernel $\kappa$ is \emph{conditionally} positive definite
if for every nonzero set $\{a_\xi\in \RR\}_{\xi\in X}$ such that
$\sum_\xi a_\xi \psi_k(\xi)=0$, $k=1,\ldots,m$, one has
\begin{equation}
\label{cpd_kernel}
\sum_{\xi,\eta\in X}a_\xi a_\eta \kappa(\xi,\eta)>0.
\end{equation}

\subsection{Interpolation}\label{interpolation}
Positive definite and conditionally positive definite kernels can be
used to interpolate a continuous function $f:\M\to \RR$, given the
data $f|_X$, by means of a function of the form
\begin{equation}
\label{interpolant}
s = \sum_{\xi\in X} a_\xi \kappa(\cdot,\xi) + \sum_{k=1}^m
b_k\psi_k,\ 
\text{where }\sum_{\xi\in X} a_\xi \psi_k(\xi)=0,\ k=1,\ldots,m.
\end{equation}
We will denote the space of such functions by $V_X$. 

In the case where the kernels are RBFs or SBFs, the space $\Pi$ is
usually taken to be either the polynomials or spherical harmonics with
degree less than some fixed number.

We now turn to the interpolation problem. Let $\Psi_k=\psi_k|_X$,
$k=1,\ldots,m$, and define the $N\times m$ matrix $\Psi=[\Psi_1\
\Psi_2 \cdots \Psi_m]$. In addition, let $a=(a_\xi)_{\xi\in X}$ and
$b=(b_1 \ \cdots b_m)^T$. The constraint condition that $\sum_\xi
a_\xi \psi_k(\xi)=0$ can now be stated as $\Psi^Ta=0$. Requiring that
$s$ interpolate a function $f\in C(\M)$ on $X$ is then
$f|_X=s|_X = Aa + \Psi b$, where $A_{\xi,\eta}=\kappa(\xi,\eta)$.
Written in matrix form, the interpolation equations are
\begin{equation}
\label{interp_matrix_form}
\underbrace{\left(\begin{array}{cc}
A&\Psi \\
\Psi^T&0_{m\times m}
\end{array}
\right) }_{\mathbf A}
\left(\begin{array}{c}
a\\
b
\end{array}\right)
= \left(\begin{array}{c}
f|_X\\
0_{m\times 1}
\end{array}\right).
\end{equation}

Using the constraint condition $\Psi^Ta=0$ and the positivity
condition (\ref{cpd_kernel}), one can easily show that the matrix
$\mathbf A$ on the left above is invertible. In addition, the
interpolation process reproduces $\Pi$; that is, if $f\in \Pi$, then
$s=f$. Finally, much is known about how well $s$ fits $f$. In many
cases, the approximation is excellent (cf. \cite{Wendland_book} and references therein).

In the sequel, we will need the \emph{Lagrange function} centered at
$\xi\in X$, $\chi_\xi\in V_X$. We define $\chi_\xi$ to be the
unique function in $V_X$ that satisfies $\chi_\xi(\eta) =
\delta_{\xi,\eta}$; that is, $\chi_\xi$ is $1$ when $x=\xi$ and $0$
when $x=\eta\in X$, $\eta\ne\xi$.

\subsection{Quadrature}\label{quadrature}
We will now develop our quadrature formula for  a
$C^\infty$, $n$-dimensional Riemannian manifold $\M$ that is a
\emph{homogeneous} space for a Lie group $\calg$
\cite{Warner-71-1}. This just means that $\calg$ acts
\emph{transitively} on $\M$: for two points $x,y\in\M$ there is a
$\gamma\in \calg$ such that $y=\gamma x$. Equivalently, $\M$ is a left
coset of $\calg$ for a closed subgroup.

$\sph^2$ is a homogeneous space for $SO(3)$. In fact, the Lie group
$\calg$ is a homogeneous space for itself. Thus, $SO(3)$ is its own
homogeneous space. Homogeneous spaces also include Stiefel manifolds,
Grassmann manifolds an many others. All spheres and projective spaces
belong to a special class known as \emph{two-point} homogeneous
spaces. Such spaces are characterized by the property that if two pairs of points $x,y$ and $x',y'$ satisfy $\d(x,y) = \d(x',y')$, then there is a group element $\gamma\in \calg$ such that $x'=\gamma x$ and $y'=\gamma y$. While spheres and projective spaces are compact, there are
also non compact spaces: $\RR^n$ and certain hyperbolic spaces also
belong to the class.

\emph{Invariance under $\calg$} plays an important here. We will
assume throughout that the kernel $\kappa$ is invariant \cite[\S
I.3.4]{vilenkin1968} under $\calg$. Moreover, we will take $d\mu$ to
be the $\calg$-invariant measure associated with the Riemannian metric
tensor for $\M$ \cite[\S I.2.3]{vilenkin1968}. Thus, for all $x,y\in
\M$, $\gamma \in \calg$, and $f\in L_1(\M)$ we have
\begin{equation}
\label{invariance_ker_def}
\kappa(\gamma x, \gamma y) = \kappa(x,y) \quad \text{and}\quad \int_\M
f(x)d\mu(x) = \int_\M f(\gamma x)d\mu(x).
\end{equation}
In particular, all SBF kernels on $\sph^n$, which have the form
$\phi(x\cdot y)$, are invariant, as is the standard measure on
$\sph^n$. The following lemma is a consequence of $\kappa$ and $d\mu$
being invariant.

\begin{lemma} \label{invariance_integral}
The integral $J(y):=\int_\M \kappa(x,y)d\mu(x)$ is
  independent of $y$.
\end{lemma}
\begin{proof}
  Take $z\in \M$ to be fixed. Because of the group action on $\M$, we
  may find $\gamma\in \calg$ such that $z=\gamma y$. From the invariance of
  $\kappa$ under $\gamma$, we have $\kappa(x,y)=\kappa(\gamma x,z)$. Hence,
  $J(y)=\int_\M \kappa(\gamma x,z)d\mu(x)$. However, the integral being
  invariant under $\calg$ then yields
\[
J(y)=\int_\M \kappa(\gamma x,z)d\mu(x)=\int_\M \kappa(x,z)d\mu(x)=J(z),
\]
which completes the proof. 
\end{proof}

Since $J(y)$ is independent of $y$, we may drop $y$ and denote it by
$J_0$, which we will do throughout the sequel.  Note that $J_0$ may be
$0$. In addition, we define these quantities.
\[
\left\{\begin{array}{rcl}
J_k&:=&\int_\M \psi_k(x)d\mu(x), \ k=1,\ldots, m; \\ [7pt]
J&:=& (J_1\cdots J_m)^T;\\
\mathbf 1 &:=&1|_X.
\end{array}\right.
\]
We point out that, for the constant function $1$ on $\M$, $\mathbf 1
=1|_X$ is the column vector in $\RR^N$ with all entries equal to 1.

The result below gives a formula for the integral of a function $s\in
V_X$ and provides a system of equations that determine the quadrature
weights. It follows the one given in \cite{sommariva_womer2005}, with
rotational invariance replaced by the invariance under $\calg$ proved
in Lemma~\ref{invariance_integral}.

\begin{proposition} \label{integral_V} Let $\kappa$ and $d\mu$ satisfy
  (\ref{invariance_ker_def}) and suppose that $c$ and $d$ are the
  $N\times 1$ and $m\times 1$ column vectors that uniquely solve the
  $(N+m)\times(N+m)$ system of equations
\begin{equation}
\label{weight_eqns}
Ac+\Psi d=J_0{\mathbf 1} \ \text{and}\ \Psi^Tc=J.
\end{equation}
If $s\in V_X$, then
\begin{equation}
\label{integra_formula_V}
\int_\M s(x)d\mu(x) = c^T s|_X=\sum_{\xi\in X}c_\xi s(\xi)
\end{equation}
In addition, if $\chi_\xi \in V_X$ is the Lagrange function centered
at $\xi$, then we also have
\begin{equation}
\label{weight_bound}
c_\xi=\int_\M \chi_\xi(x)d\mu(x) \ \text{and}\ |c_\xi| \le
\|\chi_\xi\|_{L_1(\M)}
\end{equation}
\end{proposition}

\begin{proof}
  Integrating $s(x)$ from (\ref{interpolant}) results in this chain
  of equations:
\begin{eqnarray}
  \int_\M s(x)d\mu(x) &=& \sum_{\xi\in X}a_\xi \underbrace{\int_\M 
    \kappa(x,\xi)d\mu(x)}_{J_0} + \sum_{k=1}^m b_k \underbrace{\int_\M 
    \psi_k(x)d\mu(x)}_{J_k} \label{integral_s}\\
  &=& J_0{\mathbf 1}^Ta+J^Tb \nonumber \\
  &=& (J_0{\mathbf 1}^T\  J^T)\left(\begin{array}{cc} a\\ b \nonumber 
    \end{array}\right).
\end{eqnarray}
Using (\ref{interp_matrix_form}), with $f|_X$ replaced by $s|_X$,
together with the invertibility and self adjointness of $\mathbf A$,
we obtain
\[
(J_0{\mathbf 1}^T\ J^T)\left(\begin{array}{cc} a\\
    b \end{array}\right) = \bigg\{\underbrace{{\mathbf
    A}^{-1}\left(\begin{array}{cc} J_0{\mathbf 1} \\
      J \end{array}\right)}_{\left(\begin{array}{cc} c\\
      d \end{array}\right)}\bigg\}^T \left(\begin{array}{c}
s|_X\\
0_{m\times 1}
\end{array}\right)=c^Ts|_X.
\]
Multiplying $\left(\begin{array}{cc} c\\ d \end{array}\right)$ by
$\mathbf A$ and writing out the equations for $c$, $d$ yields the
system (\ref{weight_eqns}). Combining the two previous equations then
yields (\ref{integra_formula_V}). Moreover, from
(\ref{integra_formula_V}), with $s$ replaced by $\chi_\xi$ and the
values of $s|_X$ replaced by those of $\chi_\xi$ on the set $X$, we
obtain the formula for $c_\xi$ in (\ref{weight_bound}). Finally, the
bound on the right in (\ref{weight_bound}) follows immediately from
the integral formula for $c_\xi$.
\end{proof}

As we mentioned earlier, the quadrature formula for $f\in C(\M)$ is obtained by replacing $f$ with its interpolant in $V_X$. To that end,
we define the linear functional $Q_{V_X}(f)$ that will play the role
of our quadrature operator.
 
\begin{definition} \label{quadrature_def} Let $f\in C(\M)$ and let
  $s_f\in V_X$ be the unique interpolant for $f$, so that
  $s_f|_X=f|_X$. Then, we define the linear functional
  $Q_{V_X}:C(\M)\to \RR$ via
\begin{equation}
\label{quadrature_functional}
Q_{V_X}(f):=\int_\M s_f(x)d\mu(x)=\sum_{\xi\in X} c_\xi f(\xi),
\end{equation}
where the $c_\xi$'s, the weights, are given in
Proposition~\ref{integral_V}.
\end{definition}

The invariance assumption on the kernel $\kappa$ produces a system
with the same attractive feature as one for an SBF $\phi(x\cdot
y)$. The integral $J_0$ depends only on $\kappa$ and $\M$; it is
entirely independent of $X$. This is also true of the other
$J_k$'s. As a result the equations (\ref{weight_eqns}) defining the
weight vector $c$ only depend on $X$ through function evaluations. The
integrals $J_0, J_1, \ldots,J_m$ are all known in advance and are
\emph{independent} of $X$. 

It is important to note that, \emph{without} the invariance of
$\kappa$, obtaining the system for $c$ would require computing
integrals of the form $\int_\M \kappa(x,\xi)d\mu(x)$ for each $\xi\in
X$. This follows just by looking at equation (\ref{integral_s}) in the
derivation of (\ref{weight_eqns}). This would make finding $\,c\,$ numerically very expensive, not only because each integral would have to be
computed for every $\xi\in X$, but also because the whole set would
have to be recomputed whenever $X$ was changed.

\subsubsection{Weights}\label{weights_properties}
In many cases, the weights appearing in the quadrature formula above
may be interpreted as coming from simple interpolation problems.  In
particular, if $\kappa$ is strictly positive definite and $\Pi =
\{0\}$, then the system (\ref{weight_eqns}) becomes $Ac = J_0{\mathbf
  1} $. This is the set of equations for interpolating the function
$f(x) \equiv 1/J_0$.  To obtain the interpolation problem for
quadrature with $\kappa$ being merely conditionally positive definite
with respect to $\Pi$, start with $\Psi^Tc=J$. This equation
completely determines the orthogonal projection of $c$, $c_\parallel
:= Pc$, onto the range of $\Psi$: The standard normal equations give
$c_\parallel = Pc=\Psi (\Psi^T\Psi)^{-1}\Psi^T c = \Psi
(\Psi^T\Psi)^{-1}J$. Thus, $c_\parallel$ is known. Next, let $c_\perp
:=P^\perp c$, which obviously satisfies $Pc_\perp=0$, so
$\Psi^Tc_\perp=0$, and, consequently, the following system:
\begin{equation}
\label{eq:weight_eqns_perp}
Ac_\perp +\Psi d=J_0{\mathbf 1}-A\Psi (\Psi^T\Psi)^{-1}J
\ \text{and}\ \Psi^Tc_\perp=0.
\end{equation}
This is the interpolation problem (\ref{interp_matrix_form}), with
$f_X= J_0{\mathbf 1}-A\Psi (\Psi^T\Psi)^{-1}J$. The final weights are
then 
\begin{equation}
\label{eq:weights}
c=c_\perp+\underbrace{\Psi (\Psi^T\Psi)^{-1}J}_{\displaystyle{c_\parallel}}.
\end{equation}

Note that (\ref{eq:weight_eqns_perp}) may be solved for $c_\perp$, without also
having to solve for $d$ as follows. Start by eliminating $d$ from
(\ref{eq:weight_eqns_perp}). Multiply both sides of
(\ref{eq:weight_eqns_perp}) by $P^\perp$. Since $P^\perp c_\perp =
c_\perp$ and $P^\perp \Psi=0$, we get
\begin{equation}
\label{remove_d}
P^\perp A P^\perp c_\perp =J_0P^\perp {\mathbf 1}-P^\perp A\Psi
(\Psi^T\Psi)^{-1} J.
\end{equation}
Because $\kappa$ is conditionally positive definite relative to $\Pi$,
$P^\perp A P^\perp$ is positive definite on the orthogonal complement
of the range of $\Psi$. Restricted to that space, it is
invertible. Carrying out the inverse gives us $c_\perp$. 

Often, the space $\Pi$ contains the constant function; that is,
$1 \in \Pi$. When this happens, the term with
$J_0$ drops out of (\ref{remove_d}), which then becomes 
\[
P^\perp A P^\perp c_\perp = -P^\perp Ac_\parallel = - P^\perp
A\Psi(\Psi^T\Psi)^{-1} J.
\]
This equation is homogeneous in $A$ and is therefore independent of
$J_0$. This implies that $c_\perp$ can be determined \emph{independently} of
$J_0$. 

There is another consequence of having $1 \in \Pi$. First, the column
vector $\bfone$ is in the range of $\Psi$, so $P\bfone=\bfone$. Let
$\langle c_\perp\rangle$ be the average of $c_\perp$. Since $Pc_\perp
=0$, we have $N\langle c_\perp\rangle= \bfone^Tc_\perp =
(P\bfone)^Tc_\perp= 0$. Second, since $1\in \Pi$, we also have that
the quadrature formula is exact for it, and so $Q(1) = \vol(\M)=
\bfone^Tc = N\langle c\rangle$. Using $c=c_\perp+c_\parallel$, along
with a little algebra, yields
\begin{equation}
\label{weight_averages}
\langle c_\perp\rangle = 0, \ \langle c\rangle = \langle c_\parallel
\rangle = \frac{\vol(\M)}{N}, \ \text{provided }1\in
\Pi.
\end{equation}

Signs of weights in quadrature formulas are usually nonnegative. This
is true for the polynomial quadrature formulas developed for spheres
\cite{Mhaskar-etal-01-1} and for two-point homogeneous manifolds
\cite{Brown-Dai-05-1}.  Numerical experiments by Sommariva and
Womersley \cite{sommariva_womer2005} produced negative weights for
kernels formed by restricting Gaussians to $\sph^2$. On the other
hand, their experiments involving the thin-plate splines
(cf. section~\ref{polyharmonic_kernels}) restricted to $\sph^2$
resulted in only positive weights.

Determining for what kernels and with what restrictions on $X$ all of
the weights are positive is an open problem.

\subsection{Accuracy and Stability of $Q_{V_X}$}
\label{accuracy_stability}
There are two important questions that arise concerning the quadrature
method that we have been discussing. First, how accurate is it? Second,
how stable is it?

The \emph{accuracy} of $Q_{V_X}$ depends on how well the underlying
space of functions, $V_X$ in our case, reproduces functions from the
class to be integrated. Specifically, we have the following standard
estimate:
\begin{equation}
\label{quad_op_accuracy}
\big|Q_{V_X}(f) - \int_\M f(x)d\mu(x)\big| \le \| s_f -
f\|_{L_1(\M)} \le  (\text{vol}(\M))^{1/2} \| s_f -
f\|_{L_2(\M)}.
\end{equation}
This inequality converts estimates of the accuracy of $Q_{V_X}(f)$ to
error bounds for kernel interpolation. These are known in many cases
of importance.

\emph{Stability} is related to how well the quadrature formula
performs under the presence of noise.  Lack of stability can amplify
the effect that noise in $f|_X$ will have on the value of
$Q_{V_X}(f)$. Stability also relates to performance as the number of
data sites increases. Standard one-dimensional equally spaced
quadrature formulas that reproduce polynomials can be quite unstable,
due to the well-known Runge phenomenon.

A measure of stability is the $C(\M)$ norm of the quadrature
operator. From the definition of $Q_{V_X}$ in
(\ref{quadrature_functional}) and equation (\ref{weight_bound}), we
have that
\begin{equation}
\label{quad_op_norm}
\| Q_{V_X}\|_{C(\M)} = \sum_{\xi\in X} |c_\xi| \le 
\sum_{\xi\in X} \|\chi_\xi\|_{L_1(\M)}.
\end{equation}
When $\M$ is a compact manifold, this bound can be given in terms of
the Lebesgue constant, $\Lambda_{V_X} = \max_{x\in \M} \sum_{\xi\in
  X}|\chi_\xi(x)|$. The reason is that
\begin{equation}
\label{quad_op_norm_lebesgue}
\| Q_{V_X}\|_{C(\M)} \le \sum_{\xi\in X} \|\chi_\xi\|_{L_1(\M)}\le 
\sum_{\xi\in X}\int_\M |\chi_\xi(x)|d\mu(x)\le \text{vol}(\M) \Lambda_{V_X}.
\end{equation}

Let us briefly see how noise affects $Q_{V_X}(f)$. Suppose that at
each of the sites $\xi$ we measure $f(\xi)+\nu_\xi$, where $\nu_\xi$
is a zero mean random variable. Further, we suppose that the
$\nu_\xi$'s are independent and identically distributed, with variance
$\sigma_\nu^2$. Our quadrature formula gives us $Q_{V_X}(f+\nu)$
rather than $Q_{V_X}(f)$. Because the $\nu_\xi$'s have zero mean, we
have that the mean $E\{Q_{V_X}(f+\nu)\} = Q_{V_X}(f)$. The variance of
$Q_{V_X}(f+\nu)$ is thus
\begin{equation}\label{expectation_variance_Q}
\sigma_{Q}^2 = E\{\big(Q_{V_X}(f+\nu)-Q_{V_X}(f)\big)^2\} =
E\{Q_{V_X}(\nu)^2\}.
\end{equation}
We can both calculate and estimate $\sigma_Q$:

\begin{proposition}\label{Q_standard_dev_prop}
  Let $\nu_\xi$ be independent, identically distributed, zero mean
  random variables having standard deviation $\sigma_\nu$. Then the
  standard deviation $\sigma_Q$ satisfies
\begin{equation}\label{Q_standard_dev}
  \sigma_Q^2=\sigma_\nu^2\|c\|_2^2\le \sigma_\nu^2 \|c\|_1\|c\|_\infty 
  \le \text{\rm vol}(\M) \sigma_\nu^2\,\Lambda_{V_X} \max_{\xi\in X} 
  \|\chi_\xi\|_{L_1(\M)}
\end{equation}
\end{proposition}
\begin{proof}
  We begin by evaluating the term on the right in
  (\ref{expectation_variance_Q}). to do this, we need to compute
  $E\{\nu_\xi\nu_\eta\}$. Since the $\nu_\xi$'s are i.i.d., we have
  $E\{\nu_\xi\nu_\eta\}=\sigma_\nu^2\delta_{\xi,\eta}$. It follows
  that $\sigma_{Q}^2=E\{Q_{V_X}(\nu)^2\}=
  \sigma_\nu^2\|c\|_2^2$. Moreover, $\|c\|_2^2\le \|c\|_\infty
  \|c\|_1$. Combining this with (\ref{weight_bound}) and the previous
  equation results in (\ref{Q_standard_dev}). \end{proof}

\section{Sobolev Kernels}
\label{sobolev_kernels}

There are two classes of kernels that we will discuss here. The first
class was introduced in \cite[\S~3.3]{HNW2010}, in the context of a
compact $n$-dimensional Riemannian manifold\footnote{See
  \cite{HNW2010} for a discussion of the Riemannian geometry involved,
  including metrics, tensors, covariant derivatives, \emph{etc}.} $\M$
equipped with a metric $g$. This class comprises positive definite
reproducing kernels for the Sobolev spaces $W^m_2(\M)$, as defined in
\cite{aubin1982,hebey1996}. For $\M$ a homogeneous space with $g$
being the invariant metric that $\M$ inherits from a Lie group $\calg$,
we will show below that these kernels are invariant under the action
of $\calg$.

The second class, which was introduced and studied in
\cite{HNW_3_2011}, comprises polyharmonic kernels on two-point
manifolds -- spheres, projective spaces, which include $SO(3)$, along
with a few others. These kernels are conditionally positive definite
with respect to finite dimensional subspaces of eigenfunctions of the
Laplace-Beltrami operator and they are invariant under appropriate
transformations. We will discuss these in
section~\ref{polyharmonic_kernels} below.

The Sobolev space $W^m_2(\M)$, $m\in \nats$, is defined as follows.
Let $\langle \cdot,\cdot\rangle_{g,x}$ be the inner product for a
Riemannian metric $g$ defined on $T\M_x$, the tangent space at $x\in
\M$. This inner product can also be applied to spaces of tensors at
$x$. We denote by $\nabla^k$ the $k^{th}$ order covariant derivative
associated with the metric $g$, and let $d\mu$ be the measure
associated with $g$. For $W^m_2(\M)$, define the inner product
\begin{equation}
\label{def_sn} 
  \langle f, h\rangle_{m,\M}:=\langle f,h\rangle_{W_2^m(\M)}:= 
  \sum_{k=0}^m
  \int_{\M} 
  \big\langle
  \nabla^k f,  \nabla^k h
  \big\rangle_{g,x}
  \, \dif \mu(x),
\end{equation}
and norm $\|f\|_{m,\M}^2 := \langle f,f\rangle_{m,\M}$, where
$f,h:\M\to \RR$ are assumed smooth enough for their $W^m_2$ norms to
be finite. The advantage of this definition is that it yields Sobolev
spaces that are coordinate independent and can also be defined on
measurable regions $\Omega\subseteq \M$. Using the Sobolev embedding
theorem for manifolds \cite[\S 2.7]{aubin1982}, one can show that if
$m>n/2$ these spaces are reproducing kernel Hilbert spaces, with
$\kappa_m$ being the unique, strictly positive definite reproducing
kernel for $W^m_2(\M)$; that is,
\[
f(x) = \langle f(\cdot), \kappa_m(x,\cdot)\rangle_{m,\M}
\]
In the remainder of this section we will discuss invariance,
interpolation error estimates, Lagrange functions, Lebesgue constants,
and quadrature formulas derived from $\kappa_m$.

\subsection{Invariance of $\kappa_m$}
\label{invariance_kappa_m}
We now turn to a discussion of the invariance of $\kappa_m$ under the
action of a diffeomorphism that is also an isometry. We will then
apply this to the case of a homogeneous space. Here is what we will
need.

\begin{proposition}\label{homogeneous} Let $\M$ be a compact Riemannian
  manifold of dimension $n$ with metric $g$. If $\Phi:\M\to \M$ is a
  diffeomorphism that is also an isometry, then the kernel $\kappa_m$
  satisfies $\kappa_m(\Phi(x),\Phi(y))=\kappa_m(x,y)$ and
  $d\mu(\Phi(x))=d\mu(x)$.
\end{proposition}
\begin{proof} The proof proceeds in two steps. The first is showing
  this. Let $f:\M \to \RR$ and let $f^\Phi=f\circ \Phi$ be the
  pullback of $f$ by $\Phi$. Then,
\begin{equation}
\label{pullback_invariance}
\langle \nabla^k f^\Phi, \nabla^k h^\Phi \big\rangle_{g,x}= \langle
\nabla^k f, \nabla^k h \big\rangle_{g,\Phi(x)}.
\end{equation}
We will follow a technique used in \cite[Proposition~2.4,
p.~246]{Helgason_1984} and in \cite{HNW_3_2011}. Let $(\fraku,\phi)$
be a local chart, with coordinates $u^j=\phi^j(x)$, $j=1,\ldots,n$ for
$x\in \fraku$. Since $\Phi$ is a diffeomorphism, $(\Phi(\fraku),
\phi\circ \Phi^{-1})$ is also a local chart. Let $\psi=\phi\circ
\Phi^{-1}$, and use the coordinates $v^j=\psi^j(y)$ for $y\in
\Phi(\fraku)$. The choice of coordinates has the effect of assigning
the \emph{same} point in $\RR^n$ to $x$ and $y$, provided $y=\Phi(x)$
-- i.e., $u^j(x) = v^j(y)$.  Thus, relative to these coordinates the
map $\Phi$ is the identity, and consequently, the two tangent vectors
$(\frac{\partial}{\partial v^j})_y\in T_y\M $ and
$(\frac{\partial}{\partial u^j})_x\in T_x\M$ are related via
\[
\left(\frac{\partial}{\partial v^j}\right)_{\Phi(x)} =
d\Phi_x\left(\frac{\partial}{\partial u^j}\right)_x.
\]
So far, we have only used the fact that $\Phi$ is a
diffeomorphism. The map $\Phi$ being in addition an isometry then
implies that
\[
\left\langle \frac{\partial}{\partial v^j}, \frac{\partial}{\partial
    v^k} \right\rangle_{\Phi(x)} = \left\langle
  d\Phi_x\left(\frac{\partial}{\partial u^j}\right),
  d\Phi_x\left(\frac{\partial}{\partial u^k}\right)
\right\rangle_{\Phi(x)}= \left\langle \frac{\partial}{\partial u^j},
  \frac{\partial}{\partial u^k} \right\rangle_x.
\]
The expressions on the left and right are the metric tensors at
$y=\Phi(x)$ and $x$; the equation implies that, as functions of $v$
and $u$, $g_{jk}(v) = g_{jk}(u)$. From this it follows that the
expressions for the Christoffel symbols, covariant derivatives and
various expressions formed from them also will be the same, as
functions of local coordinates. In addition, note that the local forms
for $f^\Phi$ and $h^\Phi$ at $u=u(x)$ are $f^\Phi\circ \phi^{-1}$ and
$h^\Phi\circ \phi^{-1}$, respectively, and those for $f$ and $h$ at
$v=v(y)$ are $f\circ \psi^{-1}$ and $h\circ \psi^{-1}$. At $u=u(x)$
and $v=v(y)$, $y=\Phi(x)$, $f\circ \psi^{-1}(v)=f\circ \Phi\circ
\phi^{-1}(v)=f^\Phi\circ \phi^{-1}(v)$, and similarly for $h$. Again,
this is functional equality in local coordinates, so matching partial
derivatives are also equal. Consequently, (\ref{pullback_invariance})
holds.

The equality of the coordinate forms of the metric $g_{jk}$, which was
established above, implies invariance of the Riemannian measure:
\begin{equation}
\label{measure_invariance} 
d\mu(x) = \sqrt{\det(g_{jk}(u))}d^nu = \sqrt{\det(g_{jk}(v))}d^nv =
    d\mu(\Phi(x)).
\end{equation}
Using this we see that
\begin{eqnarray}
\int_\M \langle \nabla^k f^\Phi, \nabla^k h^\phi
\big\rangle_{g,x}d\mu(x)&=&\int_\M \langle \nabla^k f, \nabla^k
h \big\rangle_{g,\Phi(x)}d\mu(x) \nonumber \\
&=& \int_\M \langle \nabla^k f, \nabla^k
h \big\rangle_{g,\Phi(x)}d\mu(\Phi(x)) \nonumber \\
&=& \int_\M \langle \nabla^k f, \nabla^k
h \big\rangle_{g,y}d\mu(y) \nonumber
\end{eqnarray}
Finally, from this and (\ref{def_sn}), we have invariance of the
Sobolev inner product:
\begin{equation}
\label{invariance_sn}
\langle f^\Phi,h^\Phi\rangle_{m,\M}=\langle f,h\rangle_{m,\M}
\end{equation}

The second step is to show that the kernel is invariant. Since
$\kappa_m$ is a reproducing kernel, we have $ f(x)=\langle
f(\cdot),\kappa_m(x,\cdot)\rangle_{m,\M} $.  By (\ref{invariance_sn}),
we also have $f(x)=\langle f^\Phi(\cdot),\kappa_m(x,\Phi(\cdot))
\rangle_{m,\M} $. Replacing $x$ by $\Phi(x)$ then yields
\[
f(\Phi(x))=f^\Phi(x)=\langle
f^\Phi(\cdot),\kappa_m(\Phi(x),\Phi(\cdot)) \rangle_{m,\M}.
\]
Finally, replacing $f$ by $f^{\Phi^{-1}}$ above gives us
\[
f(x)=\langle f(\cdot),\kappa_m(\Phi(x),\Phi(\cdot)) \rangle_{m,\M},
\]
from which it follows that $\kappa_m(\Phi(x),\Phi(y))$ is also a
reproducing kernel for $W_2^m(\M)$. But, reproducing kernels are
unique, so $\kappa_m(\Phi(x),\Phi(y))=\kappa_m(x,y)$. Thus $\kappa_m$ is
invariant.
\end{proof}

Homogeneous spaces have two properties that allow us to use the
proposition just proved. First, they inherit a Riemannian metric
invariant under the action of the Lie group $\calg$. Second, the
action of a group element produces an isometric diffeomorphism
\cite{Helgason_1984,vilenkin1968}.  These observations then yield
this:

\begin{corollary}
\label{kernel_homog_space_invariance}
Let $\M$ be a homogeneous space for a Lie group $\calg$ and suppose
that $\M$ is equipped with the invariant metric $g$ from $\calg$. Then
the reproducing kernel $\kappa_m$ for $W_2^m(\M)$ is invariant under
the action of $\gamma\in \calg$.
\end{corollary}

\subsection{Error estimates for interpolation via $\kappa_m$}
\label{error_estimates_kappa_m}

Recall that the accuracy of the quadrature formula associated with $\kappa_m$ is
directly dependent on error estimates for interpolation via
$\kappa_m$. To obtain these, we will first state a theorem that
provides estimates on functions with many zeros, quasi-uniformly
distributed over a compact manifold.

\begin{theorem}[{\cite[Corollary~A.13]{HNW_3_2011} ``Zeros Lemma"}]
  \label{zeros_lemma}
  Suppose that $\M$ is a $C^\infty$, compact, $n$-dimensional
  manifold, that $1\le p\le \infty$, $m\in \nats$, and also that $u\in
  W_p^m(\M)$. Assume that $m>n/p$ when $p>1$, and $m\ge n$ when $p=1$.
  Then there are constants $C_0=C_0(\M)$ and $C_1=C_1(m,k,\M)$ such
  that if $u|_X=0$ and $X\subset \M$ has mesh norm $h\le C_0/m^2$,
  then
\begin{equation}
\label{p_bound_W_k}
\|u\|_{W_p^k(\M)} \le C_1h^{m-k}\|u\|_{W_p^m(\M)} .
\end{equation} 
\end{theorem}

Using this ``zeros lemma'' we are able obtain an estimate for
$\|s_f-f\|_{L_2(\M)}$, provided $f\in W_2^m(\M)$ and $s_f$ is the
interpolant for $f$.

\begin{proposition}\label{interp_error_W_2_manifold}
  Let $m>n/2$, $f\in W_2^m(\M)$, and let $s_f$ be the
  $\kappa_m$-interpolant for $f$ from $V_X$. Then, with the notation
  from Theorem~\ref{zeros_lemma}, if $h\le C_0/m^2$, we have
\begin{equation}\label{interp_bound_W_2_manifold}
\|s_f-f\|_{L_2(\M)} \le C_1h^m \|f\|_{W_2^m(\M)}. 
\end{equation}
\end{proposition}

\begin{proof}
  Clearly $s_f - f \in W_2^m$ and $(s_f - f)|_X=f|_X -
  f_X=0$. Applying Theorem~\ref{zeros_lemma} to $s_f - f$ with $k=0$
  yields $\| s_f-f \|_{L_2(\M)} \le C_1h^m \| s_f - f
  \|_{W_2^m(\M)}$. Since the space $W_2^m(\M)$ is the reproducing
  Hilbert space for $\kappa_m$, the interpolant $s_f$ minimizes $\|
  g-f \|_{W_2^m(\M)}$ among all $g \in W_2^m(\M)$. Thus, taking $g=0$
  yields $\| s_f-f \|_{W_2^m(\M)}\le \| 0 -f \|_{W_2^m(\M)} = \| f
  \|_{W_2^m(\M)}$, from which the inequality
  (\ref{interp_bound_W_2_manifold}) follows immediately.
\end{proof}

The bounds in (\ref{interp_bound_W_2_manifold}) hold whether or not
$\M$ is homogeneous. In one respect, the bounds are not as strong as
we would like. The assumption that $f$ is in the reproducing kernel
space $W_2^m$ precludes estimates for less smooth $f$, in $W_2^k(\M)$,
with $k<m$. Stronger bounds do hold in important cases, though. For
example, in section~\ref{polyharmonic_kernels}, we will see that the
stronger result for $f\in W_2^k$ holds for the class of
``polyharmonic" kernels, where $\M$ can be a sphere or some other
two-point homogeneous space. We conjecture that, at least for
$\kappa_m$, the stronger result holds for general $C^\infty$ compact,
Riemannian manifolds.

\subsection{Lagrange functions and Lebesgue constants} 
\label{lagrange_lebesgu_kappa_m}
The Lagrange functions $\{\chi_\xi\}_{\xi\in X}$ associated with
$\kappa_m$ have remarkable properties. For quasi-uniform sets of
centers, Lebesgue constants are bounded independently of the number of
points and the $L_1$ norms of the Lagrange functions are nicely
controlled. Specifically, we have the result below, which holds for a
general $C^\infty$ metric $g$.

\begin{proposition}[{\cite[Theorem~4.6]{HNW2010} \cite[Proposition
    3.6]{HNSW_2_2011}}]
\label{bdd_lebesgue_sob_ker}
  Let $\M$ be a compact Riemannian manifold of dimension $n$, and
  assume $m>n/2$.  For a quasi-uniform set $X\subset \M$, with mesh
  ratio $h/q \le \rho $, there exist constants $C_\M$ and $C_{\M,m}$
  such that if $h\le C_\M/m^2$, then the Lebesgue constant
  $\Lambda_{V_X} = \max_{x\in \M} \sum_{\xi\in X}|\chi_\xi(x)|$
  associated with $\kappa_m$ satisfies 
  \[
  \Lambda_{V_X}\le
  C_{\M,m}\rho^n.
 \] 
 In addition, we have this uniform bound on the $L_1$-norm of the
 Lagrange functions:
\[
\max_{\xi\in X} \|\chi_\xi\|_{L_1(\M)} \le C_{\rho,m,\M}q^n.
\]
\end{proposition}

\begin{proof}
  In the statement of the theorem used above, we have made explicit
  the $\rho$ dependence of the bound on $\Lambda_{V_X}$ found in the
  proof of \cite[Theorem~4.6]{HNW2010}. We have also adapted the
  notation there to that used here. The bound on
  $\|\chi_\xi\|_{L_1(\M)}$ follows from \cite[Proposition
  3.6]{HNSW_2_2011}, with $s=\chi_\xi$ and $p=1$. In that proposition,
  the coefficients $A_{p,\eta}=q^{n/p}\delta_{\xi,\eta}$ and so
  $\|A_{p,\cdot}\|_p=q^{n/p}$. For $p=1$, $\|A_{1,\cdot}\|_1=q^d$, and
  so $\|\chi_\xi\|_{L_1(\M)} \le C_{\rho,m} q^n$.
\end{proof}

\subsection{Kernel quadrature via $\kappa_m$}
\label{kappa_m_quadrature}

The positive definite Sobolev kernel $\kappa_m$ is invariant under the
action of $\calg$, by
Corollary~\ref{kernel_homog_space_invariance}. This is the bare
minimum requirement for a kernel to be able to give rise to a
computable quadrature formula. The other properties of $\kappa_m$
established in the previous sections give us the result below
concerning accuracy and stability.

\begin{theorem} \label{accuracy_stabiliy_kappa_m} Let $X\subset \M$ be
  a finite set having mesh norm $\rho_X\le \rho$. Take
  $A=(\kappa_m(\xi,\eta))|_{\xi,\eta\in X}$ and suppose $f\in
  W_2^m(\M)$. Then the vector of weights in $Q_{V_X}$ is $c=J_0
  A^{-1}{\mathbf 1}$, the error for $Q_{V_X}$ satisfies $|Q_{V_X}(f) -
  \int_{\M} fd\mu| \le h^m \|f\|_{W_2^m(\M)}$, and the norm of
  $Q_{V_X}$ is bounded by $\|Q_{V_X}\|_{C(\M)} \le
  C_{\M,m}\rho^n$. Finally, the standard deviation defined in
  Proposition~\ref{Q_standard_dev_prop} satisfies the bound
  $\sigma_Q\le C_{\rho,m,\M}\sigma_\nu h^{n/2}$.
\end{theorem}

\begin{proof}
  The formula for the weights is a consequence of two things. First,
  by Corollary~\ref{kernel_homog_space_invariance} the kernel is
  invariant and so by Lemma~\ref{invariance_integral} the integral
  $\int_\M \kappa_m(x,\cdot)d\mu(x)$ is a constant, namely
  $J_0$. Second, since the kernel is positive definite,
  Proposition~\ref{integral_V} provides the desired formula for the
  weights. The error estimate is a consequence of
  (\ref{quad_op_accuracy}) and
  Proposition~\ref{interp_bound_W_2_manifold}, and the norm estimate
  follows from (\ref{quad_op_norm_lebesgue}) and
  Proposition~\ref{bdd_lebesgue_sob_ker}. Finally, the bound on
  $\sigma_Q$ is a consequence of
  Proposition~\ref{Q_standard_dev_prop},
  Proposition~\ref{bdd_lebesgue_sob_ker}, and of the fact that $\rho^n
  q^n=h^n$.
\end{proof}

There are two significant implications of this result. The first is
just what was noted in section~\ref{quadrature}; namely, the weights are
obtained by directly solving a linear system of equations.  The second is
that the measures of accuracy and stability hold for any $X$ with mesh
ratio less than a fixed $\rho$. 

There are several drawbacks. In the case of a \emph{general}
homogeneous space $\M$, the formulas for kernels $\kappa_m$ are not
yet explicitly known. This may be less of a problem when specific
cases come up in applications. Also, the error estimates, which provide
the chief measure of accuracy, hold for $f$ in the space
$W_2^m(\M)$. We would like to ``escape'' from this \emph{native space}
(reproducing kernel space) and have estimates for less smooth $f$.  As
we shall see in the next section, the situation is much improved in
the case of spheres, projective spaces, and other two-point
homogeneous manifolds. In the case of $\sph^2$, not only do we have
the requisite kernels, but in important cases we can also give fast algorithms for
obtaining the weights (cf. section~\ref{numerics}).

\section{Polyharmonic Kernels}
\label{polyharmonic_kernels}
For spheres, $SO(3)$, and other two-point homogeneous spaces
(cf. \cite[pgs.~167 \& 177]{Helgason_1984} for a list), one can use
polyharmonic kernels, which are related to Green's functions for certain differential operators. These include restrictions of
thin-plate splines \cite{HNW_3_2011}, which are useful because they
are given via explicit formulas. Many of these kernels are
conditionally positive definite, rather than positive definite.

The differential operators are polynomials in the Laplace-Beltrami
operator. Since, on a compact
Riemannian manifold $-\Delta$ is a self adjoint operator with a
countable sequence of nonnegative eigenvalues $\lambda_j <
\lambda_{j+1}$ having $+\infty$ as the only accumulation point, we can
express a polyharmonic kernel in terms of the associated eigenfunctions
$-\Delta \phi_{j,s} = \lambda_j \phi_{j,s}$, $s=1,\ldots,d_j$, $d_j$
being the multiplicity of $\lambda_j$. To make this clear, we will
need some notation.

Let $m\in \nats$ such that $m>n/2$ and let $\pi_m\in \Pi_m(\RR)$ be of
the form $\pi_m(x) = \sum_{\nu=0}^mc_\nu x^\nu$, where $c_m>0$ and let
$\L_m$ be the $2m$-order differential operator given by $\L_m =
\pi_m(-\Delta)$. We define $\J\subset \nats$ to be a finite set that
includes all $j$ for which the eigenvalue $\pi_m(\lambda_j)$ of $\L_m$
satisfies $\pi_m(\lambda_j)\le 0$. (In addition to this finite set,
$\J$ may also include a finite number $j$'s for which
$\pi_m(\lambda_j)>0$.) We say that the kernel $\kappa: \M\times\M\to
\reals$ is \emph{polyharmonic} if it has the eigenfunction expansion
\begin{equation}\label{polyharmonic_def}
\kappa(x,y):=\sum_{j=0}^\infty
\tilde \kappa_j\bigg(\sum_{s=1}^{d_j}\phi_{j,s}(x)\phi_{j,s}(y)\bigg),
\quad
\text{where }
\tilde \kappa_j:= \left\{
\begin{aligned}
\pi_m(\lambda_j)^{-1},\ & j\not\in\J, \\
\text{arbitrary, } & j\in\J.
\end{aligned}
\right. 
\end{equation}
The kernel $\kappa$ is conditionally positive definite with respect to
the space $\Pi_\J=\spn\{\phi_{j,s}:j\in \J,\ 1\le s\le d_j\}$. In the
sequel, we will need eigenspaces, Laplace-Beltrami operators and so on
for compact two-point homogenous manifolds.  Without further comment,
we will use results taken from the excellent summary given in
\cite{Brown-Dai-05-1}. This paper gives original references for these
results.

Polyharmonic kernels are special cases of \emph{zonal} functions,
which are invariant kernels that depend only on the geodesic distance
$d(x,y)$ between the variables $x$ and $y$; $d$ is  normalized so that the diameter of $\M$ is $\pi$ --- i.e., $0\le d(x,y)\le \pi$.  To see that polyharmonic kernels are zonal functions,  we
need the addition theorem on two-point homogeneous manifolds. If
$P_\ell^{(\alpha,\beta)}$ denotes the $\ell^{th}$ degree Jacobi polynomial,
normalized so that $P_\ell^{(\alpha,\beta)}(1) =
\frac{\Gamma(\ell+\alpha+1)}{\Gamma(\ell+1)\Gamma(\alpha+1)}$, then this
theorem states that for any orthonormal basis
$\{\phi_{j,s}\}_{s=1}^{d_j}$ of the eigenspace corresponding to
$\lambda_j$ we have
\[
\sum_{s=1}^{d_j}\phi_{j,s}(x)\phi_{j,s}(y) = c_{\varepsilon j}
P_{\varepsilon
  j}^{(\frac{n-2}{2},\beta)}(\cos(\varepsilon^{-1}d(x,y))), \quad
c_\ell:=\frac{(2\ell+\frac{n}{2}+\beta)\Gamma(\beta+1)
  \Gamma(\ell+\frac{n}{2}+\beta)}{\Gamma(\beta+\frac{n+2}{2})
  \Gamma(\ell+\beta+1)},
\]
where the parameters $\varepsilon$ and $\beta$, and of course
$\lambda_j$, depend on the manifold. These are listed in the table
below. In summary, every polyharmonic kernel has the form
$\kappa(x,y) = \Phi(\cos(\varepsilon^{-1}d(x,y)))$, where
\begin{equation}
\label{poly_basis_function}
\Phi(t) = \sum_{j=0}^\infty \tilde \kappa_j c_{\varepsilon
  j}P_{\varepsilon j}^{(\frac{n-2}{2},\beta)}(t), \ -1\le t\le 1.
\end{equation}

\begin{table}[ht]
\begin{center}
\begin{tabular}{|c|c|c|c|}
\hline
 {\bf Manifold	}
& $\varepsilon$	
&$\beta$
&$\lambda_j$
\\
 \hline
$\sph^n$	
& $1$
& $\frac{n-2}{2}	$
& $j(j+n-1)$\\ 

\hline
$\rproj^n$& 2& $\frac{n-2}{2}$ &$2j(2j+n-1)$
 \\ 

\hline
$\CC \mathbb{P}^n$& 1	& 0 & $j(j+\frac{n}{2})$\\
 \hline
Quaternion $\mathbb{P}^n$& 
$1$	& $1$ & $j(j+\frac{n+2}{2})$\\ 
\hline
Caley $\mathbb{P}^{16}$&1	& 3 &$j(j+12)$
\\
\hline
\end{tabular}
\caption{Parameters for compact, two-point homogeneous manifolds .}
\end{center}
\end{table}

An important class of polyharmonic kernels on $\sph^n$ are the
surface (thin-plate) splines restricted to the sphere. The surface
splines on $\RR^n$ are defined in \cite[Section 8.3]{Wendland_book};
their $\tilde \kappa_j$'s are computed in
\cite{Baxter-Hubbert-2001}. Both kernel and $\tilde \kappa_j$ are given
below in terms of the parameter $s=m+n/2$; they are also normalized to
conventions used here (see also \cite{mhaskar-etal-2010}). The first
formula is used for odd $n$ and the second for even $n$. These kernels
are conditionally positive definite of order $m$.
\begin{equation}\label{TPS}
\left.
\begin{array}{l}
\displaystyle{
\Phi_s(t)=
\left\{
\begin{array}{cl}
  (-1)^{s+\frac12}(1-t)^s, & s=m-\frac{n}{2} \in \nats+\frac12\\[5pt]
  (-1)^{s+1}(1-t)^s\log(1-t), & s=m-\frac{n}{2}\in\nats.
\end{array}
\right.}\\[18pt]
  \tilde \kappa_j=
  C_{s,n}\frac{\Gamma(j-s)}{\Gamma(j+s+n)}, \ \text{where, for even }n,
  \ j>s.
\end{array}
\right\}
\end{equation}
where the factor $C_{s,n}$ is given by
\[
C_{s,n}:=2^{s+n}\pi^{\frac{n}{2}}\Gamma(s+1)\Gamma(s+\frac{n}{2}) 
\left\{
\begin{array}{ll}
 \frac{\sin(\pi s)}{\pi} &  s=m+\frac{n}2 \in \nats+\frac12\\[5pt]
1, & s\in\nats.
\end{array}
\right.
\]
For the group $SO(3)\, (=\rproj^3)$, the following family of polyharmonic
kernels, which are similar to those above, was discussed in
\cite{Hangelbroek-Schmidt-2011}: 
\begin{equation}\label{def_so3}
\kappa(x,y) = \left(\sin\left(
\frac{\omega(y^{-1}x)}{2}\right)\right)^{2m-3},\ m\ge 2.
\end{equation}
Here $ \omega(z)$ is the rotational angle of $z\in SO(3)$. Adjusting
the normalization given in \cite[Lemma 2]{Hangelbroek-Schmidt-2011} to
that used here, the polynomial $\pi(x)$ associated with $\kappa$ is
$\pi(x) = \prod_{\nu=0}^{m-1}(x+1-4\nu^2)$. We remark that derivation of the kernels given in \cite{Hangelbroek-Schmidt-2011} used the theory of group representations along with a generalized addition formula that holds for any compact homogeneous manifold (see \cite{Gine-75}).

\begin{remark}
For two-point homogeneous manifolds, the Sobolev kernels $\kappa_m$
are in fact polyharmonic. \em The reason is that $\kappa_m$ is the
Green's function for the linear operator $\L_m = \sum_{k=0}^m
{\nabla^k}^\ast\nabla^k$. However, by \cite[Lemmas 4.2 \&
  4.3]{HNW_3_2011}, $\L_m = \pi_m(-\Delta)$, where $\pi_m(x) = x^m
+\sum_{\nu=0}^{m-1}c_\nu x^\nu$, so it has to have the form
(\ref{polyharmonic_def}). Also, it is easy to show that
$\pi_m(\lambda_j)>0$ for $j=0, 1, \ldots$, and that $\kappa_{m,j} =
\pi_m(\lambda_j)^{-1}\sim \lambda_j^{-m}$ for $j$ large. \em
\end{remark}

There are two reproducing kernel Hilbert spaces associated with a
polyharmonic kernel $\kappa$. We will take $\tilde\kappa_j>0$ for $j\in
\J$. These spaces are often called ``native spaces,'' and are denoted
by $\caln_\kappa$ and $\caln_{\kappa,\J}$. Consider the inner products
\[
\langle f,g\rangle_\kappa =
\sum_{j=0}^\infty\sum_{s=1}^{d_j}\frac{\hat f_{j,s}\overline{ \hat
    g}_{j,s}}{\tilde \kappa_j} \ \text{and}\ \langle
f,g\rangle_{\kappa,\J}= \sum_{j\not\in \J}\sum_{s=1}^{d_j}\frac{\hat
  f_{j,s}\overline{ \hat g}_{j,s}}{\tilde \kappa_j}.
\]
The first of these is the inner product of the reproducing kernel
Hilbert space for $\kappa$ itself, its native space $\caln_\kappa$,
and the second is the semi-inner product for the Hilbert space modulo
$\Pi_\J$, $\caln_{\kappa,\J}$.  It is important to note that norms for
$W_2^m$ and $\caln_\kappa$ are equivalent,
\begin{equation} \label{kernel_equiv}
\|f\|_\kappa \approx \|f\|_{W_2^m(\M)}.
\end{equation}
This follows easily from $\tilde \kappa_j \sim c_m^{-1}
\lambda_j^{-m}\sim c_m^{-1}\tilde \kappa_{m,j}$, for $j$ large.

\subsection{Interpolation with polyharmonic kernels}

If we choose $\tilde \kappa_j > 0$ for $j\in\J$, the kernel will be
positive definite, and so we can interpolate any $f\in C(\M)$. The
form of the interpolant will be
\begin{equation} \label{standard_interp}
I_Xf = \sum_{\xi\in X}a_\xi\kappa(x,\xi).
\end{equation}
with the coefficients being obtained by inverting $f|_X = Aa$, where $A
= [\kappa_(\xi,\eta)]_{\xi,\eta\in X}$ and the components of $a$ are
the $a_\xi$'s.

Let's look at error estimates, given that $f\in W_2^m(\M)$. Before we
start, we begin with the observation that (\ref{kernel_equiv}) and the
usual minimization properties of $\|\cdot \|_\kappa$ imply that
\[
\| I_Xf - f\|_{W_2^m(\M)} \le C_1\| I_Xf - f\|_\kappa \le C\|
f\|_\kappa \le C_2\|f\|_{W_2^m(\M)}.
\]
This inequality was the key to proving
Proposition~\ref{interp_error_W_2_manifold}, where the reproducing
kernel Hilbert space was itself $W_2^m(\M)$. Hence, repeating the
proof of that proposition, with the inequality changed appropriately, we
obtain the estimate
\begin{equation}
\label{general_polyharm_error}
\| I_Xf - f\|_{L_2(\M)} \le Ch^m\| f \|_{W_2^m(\M)}, \ f\in W_2^m(\M),
\end{equation}
which holds for a polyharmonic kernel (\ref{polyharmonic_def}),
provided the conditions on $X$ in
Proposition~\ref{interp_error_W_2_manifold} are satisfied. There are
no restrictions on the two-point homogeneous manifold $\M$.

There are, however, restrictions on the smoothness of $f$ -- namely,
$f$ must be at least as smooth as $\kappa$. Put another way, $f$ has
to be in the native space $\caln_\kappa$.

In the case of $\sph^n$ or $\rproj^n$, it is possible to remove this
restriction and ``escape'' from the native space of $\kappa$.  We will
first deal with $\sph^n$. From (\ref{poly_basis_function}), we see
that polyharmonic kernels on $\sph^n$ are \emph{spherical basis
  functions} (SBFs), provided coefficients with $j\in \J$ are taken to
be positive. Moreover, $\tilde \kappa_j \sim \lambda_j^{-m}$ for $j$
large. We may then apply \cite[Theorem~5.5]{Narcowich-etal-07-1}, with
$\phi, \tau,\beta, \mu \rightarrow \Phi, m, \mu, 0$, to obtain the
estimate below. (Note that for an integer $k$, the space $H_k(\sph^n)$
in \cite{Narcowich-etal-07-1} is $W_2^k(\sph^n)$ here.) This result
applies to functions \emph{not} smooth enough enough to be in $W_2^m$.

\begin{proposition}[\bf Case of $\sph^n$]\label{sphere_escape}
Let $\kappa$ be a polyhamonic kernel of the form
(\ref{poly_basis_function}), with $\deg(\pi(x))=m$, and with $\tilde
\kappa_j$, $j\in J$, chosen to be positive. Assume that the conditions
on $X\subset \sph^n$ in Proposition~\ref{interp_error_W_2_manifold}
are satisfied. If $\mu$ is an integer satisfying $m\ge \mu>n/2$ and if
$f\in W_2^\mu(\sph^n)$, then, provided $h$ is sufficiently small,
\begin{equation}\label{sphere_escape_est}
\| I_X f -f\|_{L_2(\sph^n)} \le Ch^\mu  \| f \|_{W_2^\mu(\sph^n)}.
\end{equation}
\end{proposition}

The $n$-dimensional, real projective space, $\rproj^n$, is the sphere
$\sph^n$ with antipodal points identified. Thus each $x\in \rproj^n$
corresponds to $\{x,-x\}$ on $\sph^n$. We will use this correspondence
to lift the entire problem to $\sph^n$, which is an idea used in 
\cite{graef-2011, Hangelbroek-Schmidt-2011} in connection with approximation on $SO(3)$. As
we mentioned earlier, the distance $d(x,y) = d_{\rproj^n}(x,y)$ has
been normalized so that the diameter of $\rproj^n$ is $\pi$, rather
than the $\pi/2$ one would expect from its being regarded as a
hemisphere of $\sph^n$.  In fact, the two distances are proportional
to each other. The natural geodesic distance on projective space is
simply the angle between two lines passing through $\{x,-x\}$ and
$\{y,-y\}$, with $x,y\in \sph^n$, which is just $\theta =
\arccos(|x\cdot y|)$,  $0\le \theta\le \pi/2$. The distance 
 $d_{\rproj^n}$, in which the  diameter of $\M$ is $\pi$, satisfies $d_{\rproj^n}(x,y) =2\arccos(|x\cdot y|)$. If we use this in (\ref{poly_basis_function}),
then $t = |x\cdot y|$ and $\kappa(x,y)=\Phi(|x\cdot y|)$. One can take
this farther. In the case of $\rproj^n$, the series representation for
$\Phi$ is
\[
\Phi(t) = \sum_{j=0}^\infty \tilde \kappa_j c_{2
  j}P_{2j}^{(\frac{n-2}{2},\frac{n-2}{2})}(t).
\]
Since the Jacobi polynomials $P_n^{(\frac{n-2}{2},\frac{n-2}{2})}(t)$
are even or odd, depending on whether $n$ is even or odd, and since
the series for $\Phi$ contains only polynomials with $n$ even, it
follows that $\Phi(-t)=\Phi(t)$. As a consequence we have that
$\kappa(x,y) =\Phi(x\cdot y)$. Thus, if we regard the variables $x$
and $y$ to be points on $\sph^n$, the kernel $\kappa$ is even in both
variables, and is nonnegative definite on $\sph^n$.

Our aim now is to lift $\kappa$ to a polyharmonic kernel $\kappa^*$ on
$\sph^n$, with a view toward lifting the whole interpolation problem
to $\sph^n$. To begin, note that the eignvalues for $-\Delta_{\sph^n}$
have the form $\lambda^*_j = j(j+n-1)$ and so those the $\rproj^n$
may be written as $\lambda_j = \lambda^*_{2j}$. Next, let $\pi(x)\in
\Pi_m$ be the polynomial associated with the kernel $\kappa$ in
(\ref{polyharmonic_def}) and let $\tilde
\kappa^*_j=\pi(\lambda^*_j)^{-1}$, except where $\pi(\lambda^*_j) \le
0$. For these we only require that the corresponding $\tilde
\kappa^*_j>0$. Define the function
\[
\Phi^*(t) = \sum_{j=0}^\infty \tilde
\kappa^*_jc_jP_j^{(\frac{n-2}{2},\frac{n-2}{2})}(t),
\]
which is associated with the polyharmonic kernel $\kappa^*(x,y)
=\Phi^*(x\cdot y)$ on $\sph^n$. In constructing $\Phi^*$ we have
simply added an odd function to $\Phi$. Thus,
\[
\Phi(t) = \frac12 \big(\Phi^*(t) + \Phi^*(-t) \big).
\]

Start with the centers $X$ on $\rproj^n$. Each center in $X$
corresponds to $\{\xi,-\xi\}$ on $\sph^n$, so we may lift $X$ to
$X^*=\{\pm \xi \in \sph^n: \{\xi,-\xi\}\in X \}$. Following the proof
in \cite[Theorem~1]{graef-2011}, we may show that $q_{X'}=\frac12 q_X$
and $h_{X'}=\frac12 h_X$, where $d_{\rproj^n}$ is used for $q$ and $h$
in $X$. In addition, we may, and will, identify each $f\in
C(\rproj^n)$ with an even function in $C(\sph^n)$. Now, interpolate an
even $f$ on $\sph^n$ at the centers in $X^*$. This gives us
\begin{equation} \label{S_d_interpolant}
I_{X^*}f(x) = \sum_{\zeta\in X^*} a^*_\zeta \kappa^*(x,\zeta)
\end{equation}
Note that $\kappa^*(-x,\zeta) = \Phi^*((-x)\cdot \zeta) =
\Phi^*(x\cdot (-\zeta)) = \kappa^*(x,-\zeta)$. Since both $\zeta$ and
$-\zeta$ are in $X^*$. it follows that $I_{X^*}f(-x) = \sum_{\zeta\in
  X^*} a^*_{-\zeta} \kappa^*(x,\zeta)$. Moreover,
$I_{X^*}f(-\xi)=f(-\xi)=f(\xi) =I_{X^*}f(\xi)$.  Since interpolation
is unique, and the two linear combinations of kernels agree on $X^*$,
they are equal, and so we have $I_{X^*}f(-x)=I_{X^*}f(x)$ and
$a^*_{-\zeta}=a^*_\zeta$. Because $I_{X^*}f$ is even, we have, from
(\ref{S_d_interpolant}), that $I_{X^*}f(x) =\sum_{\zeta\in X^*}
a^*_\zeta \frac12 \big(\kappa^*(x,\zeta)+\kappa^*(x,-\zeta))$. Since
$\kappa^*(x,\zeta) =\Phi^*(x\cdot\zeta)$, we have $\kappa^*(x,\zeta) +
\kappa^*(x,-\zeta) =2\kappa(x,\zeta)$. Consequently,
\[
I_{X^*}f(x) =\sum_{\zeta\in X^*} a^*_\zeta
\kappa(x,\zeta)=\sum_{\zeta\in X}a_\zeta \kappa(x,\zeta),
\]
where $a_\zeta = a^*_\zeta+a^*_{-\zeta}=2a^*_\zeta$. The sum on the
right above is the interpolant to $f$ on $\rproj^n$, $I_Xf(x)$. Thus,
$I_{X^*}f(x) = I_Xf(x)$.

\begin{corollary}[{\bf Case of $\rproj^n$}]\label{projective_escape}
Let $\kappa$ be a polyhamonic kernel of the form
(\ref{poly_basis_function}), with $\deg(\pi(x))=m$, and with $\tilde
\kappa_j$, $j\in \J$, chosen to be positive. Assume that the
conditions on $X\subset \rproj^n$ in
Proposition~\ref{interp_error_W_2_manifold} are satisfied. If $\mu$ is
an integer satisfying $m\ge \mu>n/2$ and if $f\in W_2^\mu(\rproj^n)$,
then, provided $h$ is sufficiently small,
\begin{equation}\label{projective_escape_est}
\| I_X f -f\|_{L_2(\rproj^n)} \le Ch^\mu  \| f \|_{W_2^\mu(\rproj^n)}.
\end{equation}
\end{corollary}

\begin{proof}
The metric $ds_{\rproj^n}$ we are using for $\rproj^n$ is twice the
metric inherited from the sphere; that is,
$ds_{\rproj^n}=2ds_{\sph^n}$.  From the point of view of integration,
it is easy to see that $d\mu_{\rproj^n} = 2^{n/2}d\mu_{\sph^n}$. Thus
various integrals and norms are merely changed by $d$-dependent
constant multiples -- {e.g.}, $\|I_Xf -
f\|_{L_2(\rproj^n)}=2^{1-d/2}\|I_X^*f - f\|_{L_2(\sph^n)}$.  By
Proposition~\ref{sphere_escape},
\[
\|I_X^*f - f\|_{L_2(\sph^n)}\le C_1h_{X^*}^\mu \|I_{X^*}f -
f\|_{W_2^\mu(\sph^n)}\le C_2 h_X^\mu \| I_X f
-f\|_{W_2^\mu(\rproj^n)}.
\] 
The inequality (\ref{projective_escape_est}) then follows immediately
from this and the remarks above.
\end{proof}

\subsection{Error estimates for interpolation reproducing $\Pi_\J$}

The interpolation operator associated with $\kappa$ that reproduces
$\Pi_\J = \spn\{\phi_{j,s}:j\in \J,\ 1\le s\le d_j\}$ is
\[
I_{X,\J}f = \sum_{\xi\in X}a_{\xi,\J}\kappa(x,\xi)+p_\J , \ p_\J \in
\Pi_\J, \ \text{ and } \sum_{\xi\in X}a_{\xi,\J}\phi_{j,s}(\xi) =
0,\ j\in \J.
\]
We wish to estimate the $L_2$-norm of $I_{X,\J}f-f$. To do that, we
will need to obtain various equations relating $I_{X,\J}f$, $I_Xf$,
and $p_\J$.  (Since the choice of $\tilde \kappa_j$, $j\in \J$, obviously has no effect on $I_{X,\J}$, we can and will assume that  $\tilde \kappa_j>0$, $j\in
\J$.) Let $P_\J$ be the $L_2$ orthogonal projection onto
$\Pi_\J$. Note that the interpolant $I_{X,\J}f= \sum_{\xi\in
  X}a_{\xi,\J}\kappa(x,\xi)+p_\J$ consists of two terms. It is easy to
show that the first term belongs to $(\Pi_\J)^\perp$ and, of course,
$p_\J\in \Pi_\J$. Consequently,
\begin{equation}
\label{projection_error}
P_\J(I_{X,\J}f-f)= p_\J-P_\J f.
\end{equation}
Next, because both interpolate $f$, the difference $\sum_{\xi\in
  X}(a_\xi - a_{\xi,\J})\kappa(x,\xi)$ interpolates $p_\J$. This
implies that
\[
I_Xf - I_{X,\J}f = I_Xp_\J-p_\J,
\]
From this, we have
\begin{equation}
\label{first_bnd}
\| I_{X,\J}f -f\|_{L_2(\M)} \le \|I_Xf - f\|_{L_2(\M)} + \|p_\J- I_X
p_\J\|_{L_2(\M)} 
\end{equation}
Since $p_\J- I_X p_\J|_X = 0\,$, the second term on the right can be
estimated via Theorem~\ref{zeros_lemma}, provided $h$ is sufficiently
small. Making the estimate yields this bound:
\[
 \|p_\J- I_X p_\J\|_{L_2(\M)} \le Ch^m \|p_\J- I_X p_\J\|_{W_2^m(\M)}.
\]
In addition, by (\ref{kernel_equiv}), the norm induced by $\kappa$ is
equivalent to the $W_2^m(\M)$ norm, and thus we have that $\|p_\J- I_X
p_\J\|_{W_2^m(\M)}\le C\|p_\J- I_X p_\J\|_\kappa\le C\|p_\J\|_\kappa$,
where the rightmost inequality follows from the minimization
properties of $I_Xp_\J$ in the norm $\|\cdot\|_\kappa$. From the
definition of the $\kappa$ norm, where we have assumed $\tilde
\kappa_j=1$, $j\in \J$, it is easy to see that $\|p_\J\|_\kappa =
\|p_\J\|_{L_2(\M)}$. Combining the various inequalities then yields
\[
\|p_\J- I_X p_\J\|_{L_2(\M)} \le Ch^m \|p_\J\|_{L_2(\M)}.
\]
Using this on the right in (\ref{first_bnd}) gives us
\begin{equation}
\label{second_bnd}
\| I_{X,\J}f -f\|_{L_2(\M)} \le 
\|I_Xf - f\|_{L_2(\M)} + Ch^m \|p_\J\|_{L_2(\M)}.
\end{equation}
Furthermore, employing (\ref{projection_error}) in conjunction with
this inequality, we see that
\[
\| I_{X,\J}f -f\|_{L_2(\M)} \le 
\|I_Xf - f\|_{L_2(\M)} + Ch^m  \big(\|I_{X,\J}f-f\|_{L_2(\M)}+
\|P_\J f\|_{L_2(\M)}\big)
\]
Choosing $h$ so small that $Ch^m<\frac12$ and then manipulating the
expression above, we obtain this result.

\begin{lemma} \label{final_bnd_poly_reprod}
Let $f\in L_2(\M)$.  For all $h$ sufficiently small, we have 
\begin{equation}\label{third_bnd}
\| I_{X,\J}f -f\|_{L_2(\M)} <
2\|I_Xf - f\|_{L_2(\M)} + Ch^m  
\|P_\J f\|_{L_2(\M)}.
\end{equation}
\end{lemma}

\noindent We now arrive at the error estimates we seek.

\begin{theorem}\label{interp_error_reproducion}
Let $\M$ be a two-point homogeneous manifold and let $\kappa$ be a
polyhamonic kernel of the form (\ref{poly_basis_function}), with
$\deg(\pi(x))=m$. Assume that the conditions on $X\subset \M$ in
Proposition~\ref{interp_error_W_2_manifold} are satisfied. If $\mu$ is
an integer satisfying $m\ge \mu>n/2$, then,
provided $h$ is sufficiently small,
\begin{equation}\label{interp_error_reproduce}
\| I_{X,\J} f -f\|_{L_2(\M)} \le 
\left\{
\begin{array}{cl}
Ch^m \| f \|_{W_2^m(\M)}, & \text{\rm all } \M, 
\ \text{and } f\in
W_2^m(\M), \\ [7pt] 
Ch^\mu \| f \|_{W_2^\mu(\M)}, & \M = \sph^n,
\rproj^n, \text{and }f\in W_2^\mu(\M).
\end{array}
\right.
\end{equation}
\end{theorem}

\begin{proof}
In (\ref{third_bnd}), we have $\|P_\J f\|_{L_2(\M)} \le \|
f\|_{L_2(\M)} \le \| f\|_{W_2^\mu(\M)}$; in addition, the various
bounds on $\| I_X f -f\|_{L_2(\M)}$ follow from
(\ref{general_polyharm_error}), (\ref{sphere_escape_est}), and
(\ref{projective_escape_est}). Combining these yields
(\ref{interp_error_reproduce}).
\end{proof}

\subsubsection{Optimal convergence rates  for interpolation via surface splines} \label{superconvergence}

For spheres, the interpolants constructed from the restricted thin-plate splines defined in (\ref{TPS}) converge at double the rates discussed above -- namely, $\calo (h^{2m})$ rather than $\calo (h^{m})$ --, provided $f$ is smooth enough. This also applies in $SO(3)$ for the kernels given in (\ref{def_so3}). The spaces $\Pi_\J$ being reproduced here are, for $\sph^n$, the spherical harmonics of degree $m-1$ or less; and for $SO(3)$, the Wigner D-functions for representations of the rotation group, again having order $m-1$ or less. The precise result is stated below.

\begin{proposition}[{\cite[Corollaries 5.9 \& 5.10]{HNW_3_2011}}] \label{sphere_result}
Suppose $m>n/2$ and let $I_{X,\J}$ denote the interpolation operator corresponding to 
the restricted surface spline defined in (\ref{TPS}) for $s=m-n/2$. Then, there is a constant $C$ depending on $\rho, m$ and $n$ such that, for a sufficiently dense set $X \subset \sph^n$,  and for $f\in C^{2m}(\sph^n)$,  the following holds:
\[
\|I_{X,\J}f - f\|_{\infty} \le C h^{2m} \|f\|_{C^{2m}(\sph^n)},
\]
Similarly, for the kernels in (\ref{def_so3}), if $f\in C^{2m}(SO(3))$, then $\| I_{X,\J}f - f\|_{\infty} \le C h^{2m} \|f\|_{C^{2m}(SO(3))}$.
\end{proposition}

\subsection{Lagrange functions and Lebesgue constants}

If $\chi_\xi(x)$ is the Lagrange function associated with a kernel
$\kappa$ and a space $\Pi_\J$, then, by (\ref{weight_bound}), the
weights in the quadrature formula have the form $c_\xi = \int_\M
\chi_\xi(x)d\mu(x)$, $\xi\in X$. Our aim is to use properties of
$\chi_\xi$ to obtain bounds on these weights. Before we do this, we
will need the following decay estimates for $\chi_\xi$.

\begin{proposition}[{\cite[Theorem~5.5]{HNW_3_2011}}] \label{general_decay}
Suppose that $\M$ is an $n$-dimensional compact, two-point homogeneous
manifold and that $\kappa$ is a polyharmonic kernel, with
$\deg(\pi(x))=m$, where $m>n/2$.  There exist positive constants
$h_0$, $\nu$ and $C$, depending only on $m$, $\M$ and the operator
$\L_m=\pi(-\Delta)$ so that if the set of centers $X$ is quasi-uniform
with mesh ratio $\rho$ and has mesh norm $h\le h_0$, then the Lagrange
functions for interpolation by $\kappa$ with auxiliary space
$\Pi_{\J}$ satisfy
\begin{equation}\label{general_pointwise_lagrange_bound}
 |\chi_{\xi}(x)| \le C \rho^{m-n/2} \max\left(\exp\left(-
 \frac{\nu}{h} d(x,\xi) \right), h^{2m}\right).
 \end{equation}

\end{proposition}

There are two results that we want to obtain. First, we want to
estimate the size of each weight in terms of the mesh norm $h\,$; and
second, we want to estimate $\sum_{\xi\in X}|c_\xi|$.

We only need to estimate $\|\chi_\xi\|_{L_1(\M)}$, since $|c_\xi|\le
\|\chi_\xi\|_{L_1(\M)}$. Our approach is to divide the manifold into a
ball $B=B_{\xi,R_h}$, having radius $R_h$ and center $\xi$, and its
complement $B^\complement$. The radius $R_h$ is the ``break-even''
distance in (\ref{general_pointwise_lagrange_bound}); it is obtained
by solving $\exp\left(- \frac{\nu}{h} R_h \right)= h^{2m}$ for
$R_h$. The result is $R_h= \frac{2m}{\nu}h |\log(h)|$. By
(\ref{general_pointwise_lagrange_bound}), we obtain
\[
\int_{B_0^\complement}\big| \chi_\xi(x)\big|d\mu(x) \le
C\rho^{m-n/2}\mu(\M)h^{2m}.
\]
Next, again by (\ref{general_pointwise_lagrange_bound}), we have
that
\[
\int_B |\chi_\xi(x)| d\mu(x) \le
C'\rho^{m-n/2}\int_0^{R_h}e^{-r\nu/h}r^{n-1}dr
<C'\rho^{m-n/2}\underbrace{\int_0^\infty e^{-r\nu/h}r^{n-1}dr}_{(n-1)!(h/\nu)^n}
\]
Consequently, $ |c_\xi|\le \|\chi_\xi\|_{L_1(\M)}\le
C\rho^{m-n/2}(h/\nu)^n(1+h^{2m-n}\nu^n) $.  Because $2m>n$, for $h$
small enough, it follows that
\begin{equation}\label{simple_weight_bound}
|c_\xi|\le  \|\chi_\xi\|_{L_1(\M)} \le C'\rho^{m-n/2}h^n,
\end{equation}
where $C'=C'(\M,\kappa,\J)$. 
 
The next result concerns the boundedness of the Lebesgue constant, which played a significant role in section~\ref{accuracy_stability} in the analysis of the stability of the quadrature operator.
 
\begin{proposition}[{\cite[Theorem~5.6]{HNW_3_2011}}] 
\label{polyharmonic_lebesgue}
Let the notation be the same as that in
Proposition~\ref{general_decay}. If $h\le h_0$, then the Lebesgue
constant, $\Lambda_{X,\J,\kappa} = \max_{x\in \M} \sum_{\xi\in X}
|\chi_{\xi}(x)|$, associated with $X$, $\J$, and $\kappa$, satisfies
the bound $\Lambda_{X,\J,\kappa}\le C$, where $C$ depends only on
$m$, $\rho$, and $\M$.
\end{proposition}

\subsection{Quadrature via polyharmonic kernels}
\label{polyharmonic_kappa_quadrature}

In this section, we will employ the various properties of polyharmonic
kernels, which are of course invariant, to obtain results concerning
accuracy and stability of the associated quadrature formulas. 

\begin{theorem} \label{accuracy_stabiliy_polyharmonic} 
Suppose that $\M$ is an $n$-dimensional compact, two-point homogeneous
manifold and that $\kappa$ is a polyharmonic kernel, with
$\deg(\pi(x))=m$, where $m>n/2$.  Let $X\subset \M$ be a finite set
having mesh ratio  $\rho_X\le \rho$, $\Pi_\J = \spn\{\phi_{j,s}:j\in
\J, \, 1\le s \le d_j\}$, and $Q_{V_X}$ be the corresponding quadrature operator
given in Definition~\ref{quadrature_def} . The norm of $Q_{V_X}$ is bounded,
$\|Q_{V_X}\|_{C(\M)} \le \mu(\M)\Lambda_{X,\J,\kappa}\le C_{m,\rho,\M}$, and
the error satisfies the estimates
\begin{equation}\label{quad_error_reproduce}
\big| Q_{V_X}(f) - \int_{\M} fd\mu \big| \le 
\left\{
\begin{array}{cl}
Ch^m \| f \|_{W_2^m(\M)}, & \text{\rm all } \M, 
\ \text{and } f\in
W_2^m(\M), \\ [7pt] 
Ch^\mu \| f \|_{W_2^\mu(\M)}, & \M = \sph^n,
\rproj^n, \text{and }f\in W_2^\mu(\M), \ n/2<\mu\le m.
\end{array}
\right.
\end{equation}
Finally, the standard deviation $\sigma_Q$ from
Proposition~\ref{Q_standard_dev_prop} satisfies $\sigma_Q\le
C\sigma_\nu h^{n/2}$, where $C=C(\rho,m,\M,\J)$. 
\end{theorem}

\begin{proof}
The norm estimate follows from (\ref{quad_op_norm_lebesgue}) and
Proposition~\ref{polyharmonic_lebesgue}. The error estimate is a
consequence of (\ref{quad_op_accuracy}) and
Theorem~\ref{interp_error_reproducion}. Finally, the bound on
$\sigma_Q$ is a consequence of Proposition~\ref{Q_standard_dev_prop},
Proposition~\ref{polyharmonic_lebesgue}, and the bound on
$\|\chi_\xi\|$ in (\ref{simple_weight_bound}).
\end{proof}

At present, the most important compact two-point homogeneous manifolds are spheres and projective spaces (especially $\sph^2$ and $SO(3))$. As we discussed in section~\ref{superconvergence}, the restricted thin-plate splines (\ref{TPS}) on $\sph^n$ and similar kernels (\ref{def_so3}) on $SO(3)$ give interpolants with optimal convergence for smooth target functions. This also is reflected in the accuracy of the corresponding quadrature formulas:

\begin{corollary}\label{TPS_quadrature}
Let $X$, $\rho_X$, $\rho$, $m$, $\Pi_\J$  be as in Theorem~\ref{accuracy_stabiliy_polyharmonic}. Take $\M=\sph^n$ or $SO(3)$ and $Q$ to be the quadrature operator corresponding to the restricted surface spline defined in (\ref{TPS}) or to that for the  polyharmonic kernel in (\ref{def_so3}). If $X$ is a sufficiently dense in $\M$, then there is a constant $C=C(m,n,\rho)$ such that for a sufficiently dense set $X \subset \sph^n$ and for all $f\in C^{2m}(\sph^n)$,  
\[
 \big|Q(f) - \int_{\M} f d\mu \big| \le C h^{2m} \|f\|_{C^{2m}(\sph^n)},\ \text{where }C=C(m,n,\rho),
\]
in addition to the bounds on $\|Q\|$ and $\sigma_Q$ from Theorem~\ref{accuracy_stabiliy_polyharmonic} holding.
\end{corollary}

Because of applications in physical sciences and engineering, $\sph^2$ is undoubtedly the most important of the manifolds treated here.  In the next section, we give some numerical examples for $\sph^2$ and the $m=2$ surface spline $\Phi(t) = (1-t)\log(1-t)$ (i.e. the thin plate spline $r^2\log(r^2)$ restricted to $\sph^2$) that validate the above theory.

\section{Numerical results for $\sph^2$}
\label{numerics}
We begin with a brief overview of how the quadrature weights can be computed in an efficient manner for the $m=2$ surface spline $\Phi(t) = (1-t)\log(1-t)$ using the local Lagrange preconditioner developed in~\cite{FHNWW2012}.  This is followed by a description of the nodes used in the numerical experiments and some properties of the resulting surface spline quadrature weights and their stability.  Finally, we give some results validating the error estimates from the previous section.

\subsection{Computing the quadrature weights}
The $m=2$ restricted surface spline kernel is conditionally positive definite of order 1 and the finite dimensional subspace $\Pi$ associated with it consists of all spherical harmonics of degree $\leq 1$, i.e. $\Pi: = \text{span}\{Y_{0,0}, Y_{1,0}, Y_{1,1}, Y_{1,2}\}$, where $Y_{\ell,k}$ is the degree $\ell$ and order $k$ spherical harmonic and $0\le k\le 2\ell+1$.  Given a set $X = \{x_j\}_{j=1}^{N}$ of distinct nodes on $\sph^2$, the quadrature weights $c$ for this kernel can be computed by first solving \eqref{eq:weight_eqns_perp} for $c_{\perp}$ and then computing $c$ via \eqref{eq:weights}.  In these equations, $A_{i,j} = (1-x_i \cdot x_j)\log(1-x_i\cdot x_j)$, $i,j=1,\ldots,N$, and $\Psi$ is the $N$-by-$4$ matrix with columns $\Psi_{i,1} = Y_{0,0}(x_i)$ and $\Psi_{i,k+2} = Y_{1,k}(x_i)$, for $i=1,\cdots,N$, $k=0,1,2$.  Additionally, $J=\begin{bmatrix}4\pi & 0 & 0 & 0\end{bmatrix}^T$ and $J_0 = 2\pi(4\log(2)-1)$.   

Since $\Psi$ only has 4 columns, computing $\Psi(\Psi^T \Psi)^{-1}J$ in \eqref{eq:weight_eqns_perp} can be done rapidly using, for example, QR decomposition.   Thus, the bulk of the computational effort in computing the quadrature weights is in solving for $c_{\perp}$ in \eqref{eq:weight_eqns_perp}.  Since the matrix $A$ is dense, direct methods cannot be realistically applied for large $N$, and one must then resort to iterative methods.  However, for iterative methods to be useful, one must apply an effective preconditioner to the system.  In~\cite{FHNWW2012}, we developed a powerful preconditioner for \eqref{eq:weight_eqns_perp} based on \emph{local Lagrange functions} and combined it with the generalized minimum residual (GMRES) iterative method ~\cite{SaadSchultz}.  The basic idea of the preconditioner is, for every node $x_j\in X$, to compute the surface spline interpolation weights for a small subset of nodes about $x_j$ consisting its $p = M(\log N)^2$ nearest neighbors, where $M$ is suitably chosen constant. The data for each interpolant is taken to be cardinal about $x_j$.  This is similar to the preconditioner used in~\cite{Faul-Powell-99-1} for interpolation on a 2-D plane, however, in that study the number of nodes in the local interpolants did not grow with $N$ and it was observed that the preconditioner broke down as $N$ increased.  As demonstrated in~\cite{FHNWW2012}, by allowing the nodes to grow very slowly with $N$ the preconditioner remained effective and the total number of iterations required by GMRES to reach a desired tolerance did not increase with $N$.   We refer the reader to~\cite{FHNWW2012} for complete details on the construction of the preconditioner.

In the numerical results that follow, we used the preconditioned GMRES technique from~\cite{FHNWW2012} with $p = 2 \lceil (\log N)^2\rceil$ and a relative tolerance of $10^{-12}$ to solve for $c_{\perp}$ in \eqref{eq:weight_eqns_perp}.  We then used this in \eqref{eq:weights} to find $c$.  Table \ref{tbl:iterations} lists the number of GMRES iterations required to compute $c_{\perp}$ for three different quasi-uniform node families, which are described in the next section.  We see from the table that the number of iterations remains fairly for constant as $N$ grows for all three families of nodes.  

\begin{table}[h]
\centering
\begin{tabular}{||r|c||r|c||r|c||}
\hline
\hline
\multicolumn{2}{||c||}{Icosahedral} & \multicolumn{2}{c||}{Fibonacci} & \multicolumn{2}{c||}{Min. Energy} \\
N & Iterations & N & Iterations & N & Iterations \\
\hline
\hline
$2562$ & $8$ & $2501$ & $9$ & $2500$ & $9$ \\
$10242$ & $7$ & $10001$ & $8$ & $10000$ & $8$ \\
$23042$ & $7$ & $22501$ & $11$ & $22500$ & $7$ \\
$40962$ & $7$ & $40001$ & $8$ & $40000$ & $8$ \\
$92162$ & $6$ & $62501$ & $10$ & $62500$ & $7$ \\
$163842$ & $7$ & $90001$ & $10$ & $90000$ & $7$ \\
$256002$ & $6$ & $160001$ & $8$ & $160000$ & $7$ \\
$655362$ & $6$ & $250001$ & $10$ & $250000$ & $7$  \\
\hline
\hline
\end{tabular}
\caption{Number of GMRES iterations required to compute $c_{\perp}$ in  \eqref{eq:weight_eqns_perp} using the preconditioned iterative method developed~\cite{FHNWW2012} for determining the $m=2$ surface spline.  The quadrature weights are then computed from  \eqref{eq:weights}.  In all cases, the relative tolerance of the GMRES method was set to $10^{-12}$.\label{tbl:iterations}}
\end{table}

We conclude by noting that as part of the iterative method, matrix-vector products involving $A$ must be computed.  Since $A$ is dense, this requires $\calo(N^2)$ operations per matrix-vector product.  In the computations performed for this study, these products were computed directly, making the overall cost of the weight computation $\calo(N^2)$.  In a follow up study we will explore the use of fast, approximate matrix-vector products using the algorithm described in~\cite{Keiner:2006:FSR:1152729.1152732}.  By using this algorithm, it may be possible to reduce the total cost of computing the weights (or a surface spline interpolant) to $\calo(N\log N)$.

\subsection{Nodes, weights, and stability}
We consider three quasi-uniform families of nodes for the numerical experiments.  The first is the icosahedral nodes, which are obtained from successive refinement of the 20 spherical triangular faces formed from the icosahedron.  
The second are the Fibonacci (or \emph{phyllotaxis}) nodes, which mimic certain plant behavior in nature (see, for example,~\cite{Gonzalez:2010} and the references therein).  The third are the quasi-minimum energy nodes, which are obtained by arranging the nodes so that their Riesz energy is near minimal~\cite{HarinSaff04}.  In the examples below, a power of 3 was used in the Riesz energy and the nodes were generated using the technique described in~\cite{BorodachovHardinSaff:2012}.  The mesh norm for all three of these families satisfies $h \sim \frac{1}{\sqrt{N}}$, where $N$ is the total number of nodes.  The mesh ratio $\rho$ stays roughly constant for the Fibonacci and quasi-minimum energy nodes as $N$ increases.  For the icosahedral points $\rho$ grows slowly with $N$ since the spacing of the nodes decreases faster towards the vertices of the triangles of the base icosahedron than at the centers~\cite{Saff-Kuijlaars-97-1}.  However, in the numerical examples that follow, this increase seems to be of little concern.  All three of these families of nodes are quite popular in applications; see, for example~\cite{Giraldo:1997,StuhnePeltier:1999,Ringler:2000GeodesicGrids,Majewski:2002GME} for the icosahedral nodes,~\cite{SwinbankPurser:2006,SlobbeSimonsKlees:2012,HuttigKai:2008} for the Fibonacci nodes, and~\cite{WrightFlyerYuen,flyer_wright2009,FlyerLehtoBlaiseWrightStCyr2012,SWFK2012} for the quasi-minimum energy nodes.  Quadrature over these nodes also plays an important role in applications, for example, for computing the mass of a certain quantity, the energy for a certain process, or the spectral decomposition of some data.  

It should be noted that previous studies have been devoted to developing quadrature formulas and error estimates for icosahedral~\cite[Ch. 5]{atkinson2012spherical} and Fibonacci~\cite{HannayNye:2004} nodes.  However, these results rely on the specific construction of the node sets and cannot be applied to more general quasi-uniformly distributed nodes such as the quasi-minimum energy nodes.

\begin{figure}[t!]
\centering
\begin{tabular}{cc}
\includegraphics[width=0.48\textwidth]{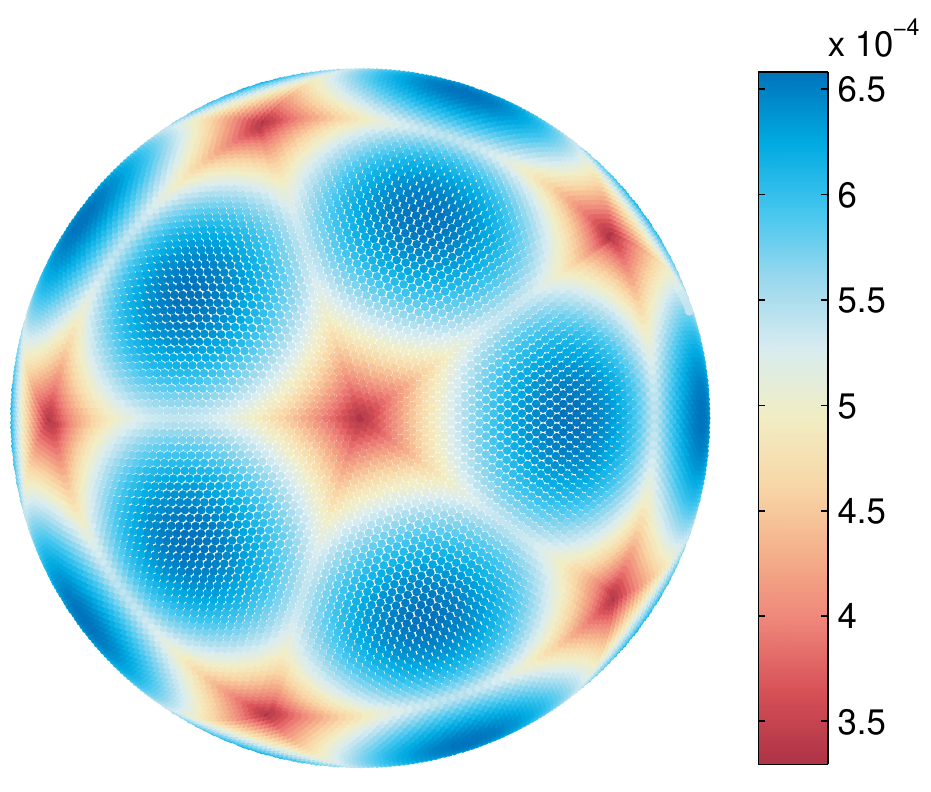} &
\includegraphics[width=0.48\textwidth]{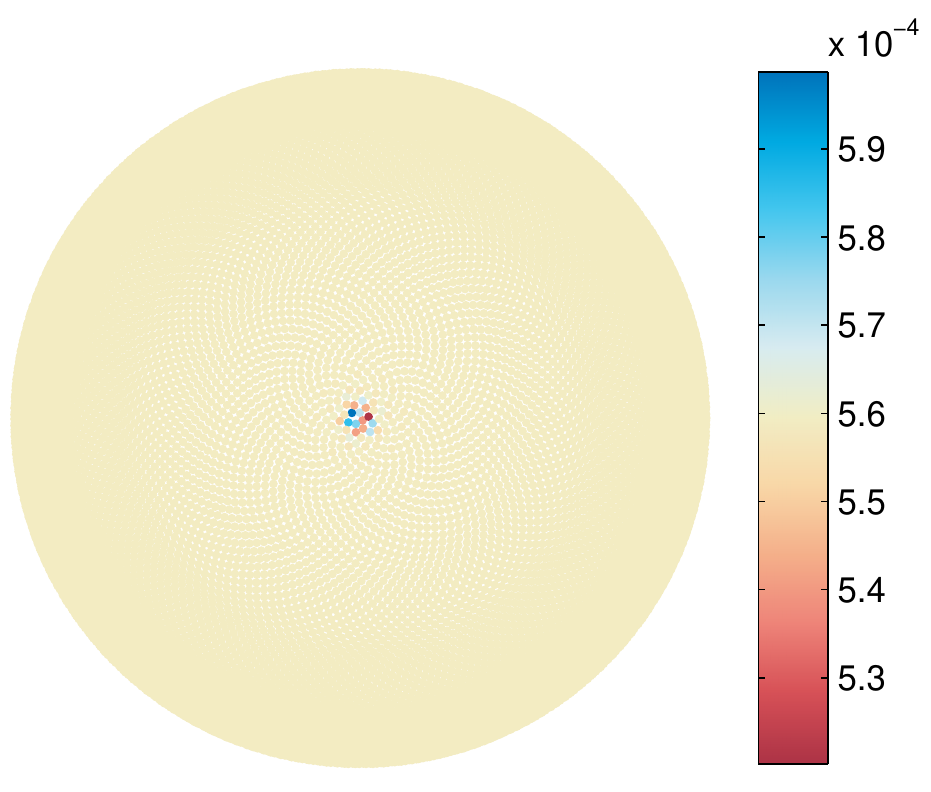} \\
(a) Icosahedral nodes and weights, $N=23042$ & (b) Fibonacci nodes and weights, $N=22501$ \\
\multicolumn{2}{c}{\includegraphics[width=0.48\textwidth]{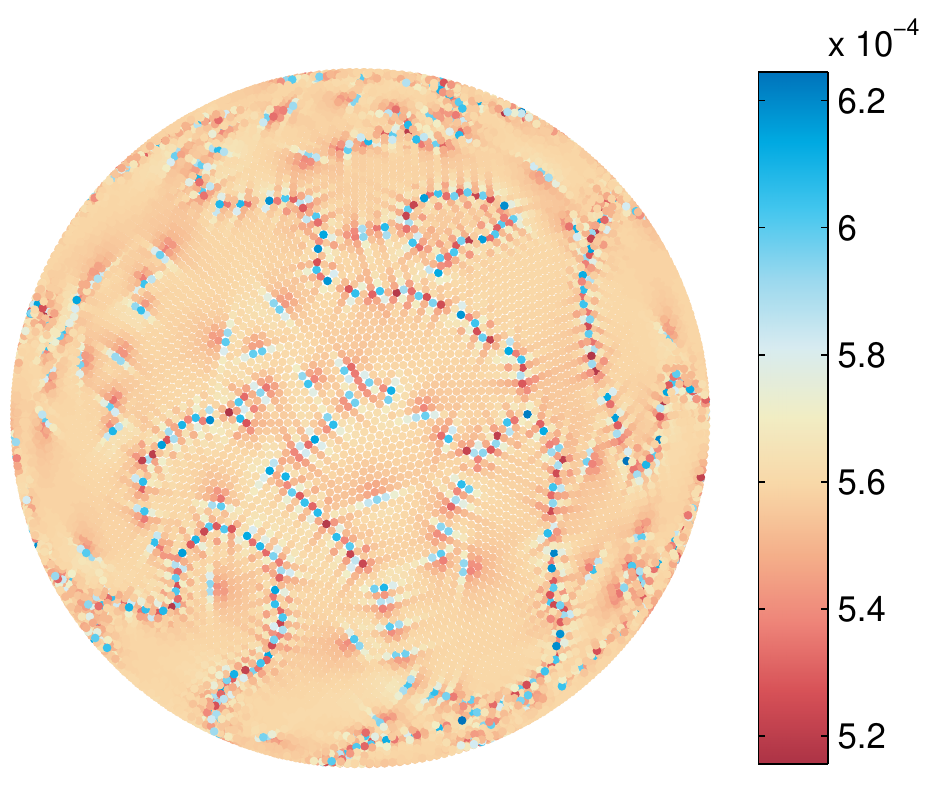}} \\
\multicolumn{2}{c}{(c) Quasi-minimum energy nodes and weights, $N=22500$}
\end{tabular}
\caption{Visualization of the different families of nodes used in the numerical experiments and the corresponding quadrature weights for the $m=2$ spherical spline.  The nodes are plotted on $\sph^2$ using an orthographic projection about the north pole.  Each node has been given a color corresponding to the value of the quadrature weight for that node.\label{fig:quad_wghts}}
\end{figure}

Figures \ref{fig:quad_wghts} (a)--(c) show examples of the different families of nodes and provide a visualization of the values of the corresponding quadrature weights for the $m=2$ surface spline.  The geometric pattern of the icosahedral nodes in part (a) of this figure are clearly reflected in the values of the corresponding quadrature weights.  This is also true of the Fibonacci nodes in part (b), which have a slight clustering near their seed value (in this case the north pole), but then are quasi-uniformly distributed.  There are no clear patterns for the minimum energy nodes in part (c) of Figure \ref{fig:quad_wghts}, as these are not distributed in a discernible pattern.  Looking at the color bars in each of the plots we see that the range of values of the weights is similar between the different node families and comparable to $1/N$,  and are also positive.

\begin{figure}[ht]
\centering
\includegraphics[width=0.48\textwidth]{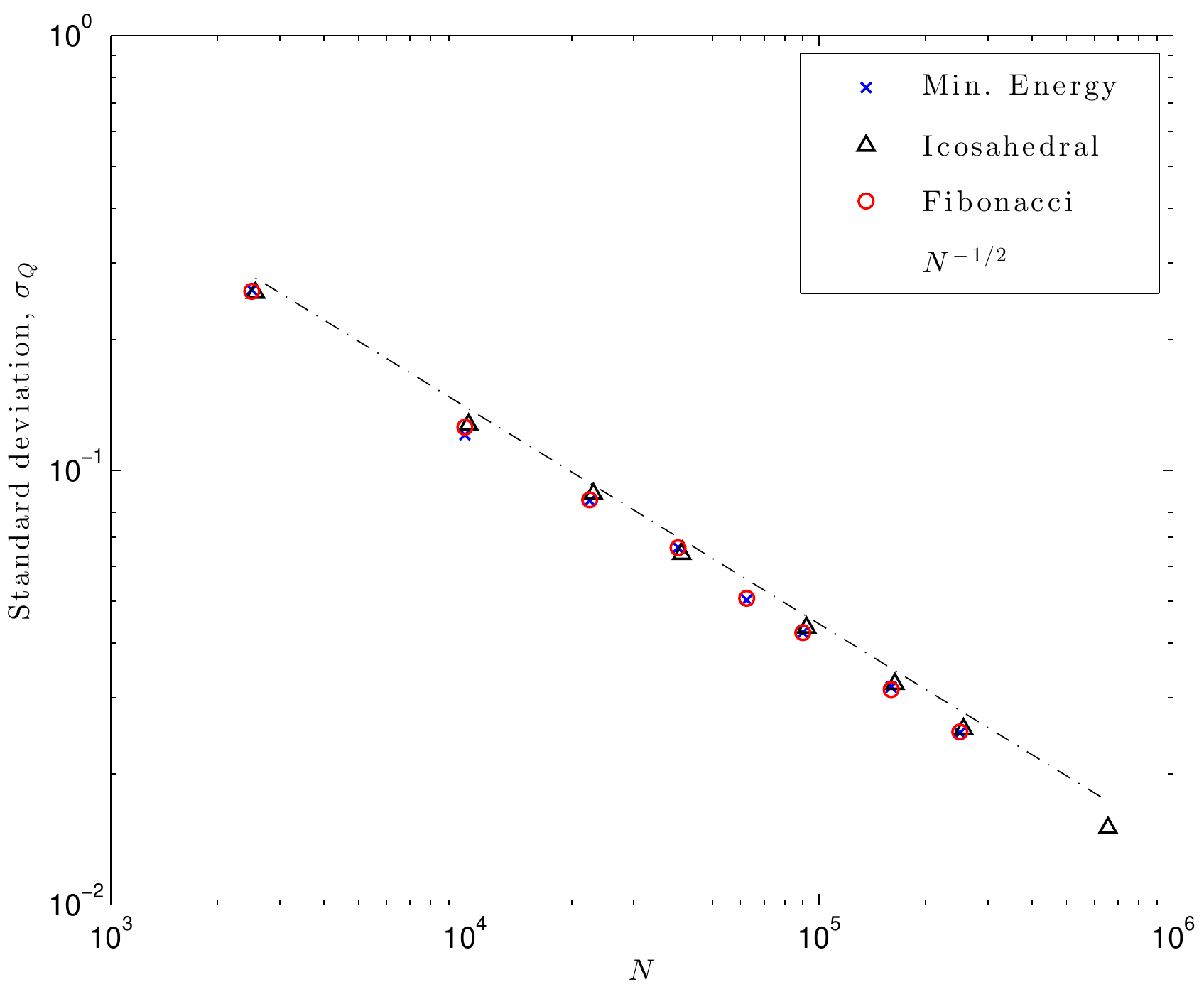}
\caption{Estimate of the standard deviation $\sigma_Q$ in \eqref{expectation_variance_Q} for the different families of nodes.  The estimate is based on sample sets of 500 $N$-point quadratures of different independent, identically, distributed, zero mean (quasi) random data with a standard deviation of 1.   This confirms the stability estimate in the last part of Theorem \ref{accuracy_stabiliy_polyharmonic}. \label{fig:stability}}
\end{figure}

To estimate $\sigma_Q$ in \eqref{expectation_variance_Q}, and hence the stability of the quadrature weights in the presence of noise, we performed the following experiments.   For each family of nodes, a value of $N$ was selected and then a sample set was generated consisting of 500 values of the $N$-point quadrature of different independent, identically, distributed, zero mean (quasi) random data.  The standard deviation of the sample sets was then computed to estimate $\sigma_Q$.  The results are plotted in Figure \ref{fig:stability} as a function of $N$ for each of the three families of nodes.  Comparing the results to the dashed line on the plot, we see that the estimated $\sigma_Q$ decreases like $\calo(N^{-1/2})$, or like $\calo(h)$ since $h \sim N^{-1/2}$ for these families of nodes.  This is in perfect agreement with the rate predicted by the last part of Theorem \ref{accuracy_stabiliy_polyharmonic} for the surface splines. 

\subsection{Convergence results}
Two target functions of different smoothness were used to test the error estimates of Theorem \ref{accuracy_stabiliy_polyharmonic} and Corollary \ref{TPS_quadrature}.  The target functions were chosen so that the Funk-Hecke formula (see, for example,\cite[\S 2.5]{atkinson2012spherical}) could be used to determine their exact integral over $\sph^2$.  Letting $x,x_c\in \sph^2$ and $g$ be a zonal kernel, i.e. $g(x,x_ci) = g(x\cdot x_c)$, such that $g \in L^{1}[-1,1]$, the Funk-Hecke formula gives the following result:
\begin{align}
\int_{\sph^2} g(x \cdot x_c)Y_{\ell,k}(x) d\mu(x) = \frac{4\pi a_{\ell}}{2\ell+1} Y_{\ell,k}(x_c),
\label{eq:funk_hecke}
\end{align}
where $Y_{\ell,k}$ is any degree $\ell$, order $k$ spherical harmonic, and $a_{\ell}$ is the $\ell$th coefficient in the Legendre expansion of $g(t)$.  

The following two kernels were used in constructing the target functions:
\begin{align}
g_1(t) &= -(2 - 2t)^{1/4},\\
g_2(t) &= \frac{1-\varepsilon^2}{(1 + \varepsilon^2 - 2\varepsilon t)^{\frac32}}\quad (0 < \varepsilon < 1),
\end{align}
which are known as the \emph{potential spline} kernel of order $1/4$ and the \emph{Poisson} kernel, respectively.  The Legendre expansion coefficients for these kernels are given as follows (see~\cite{Baxter-Hubbert-2001}):
\begin{align}
g_1:\quad a_{\ell} &= \frac{(-1)^{\ell+1}\sqrt{2}(\Gamma\left(\frac54\right))^2(2\ell+1)}{\Gamma\left(\frac54 - \ell\right)\Gamma\left(\frac94+\ell\right)},\\
g_2:\quad a_{\ell} &= (2\ell+1)\varepsilon^{\ell}.
\end{align}
The smoothness of these kernels is of course determined by the decay rate of the Legendre coefficients $a_{\ell}$.  For $g_1$, we have $a_{\ell} \sim \ell^{-3/2}$, which means $g_1$ belongs to every Sobolev space $W_2^{\mu}(\sph^2)$ with $\mu < \frac52$.   While for $g_2$ the Legendre coefficients decay exponentially fast, which means $g_2\in C^{\infty}(\sph^2)$, analytic in fact.

\begin{figure}[t]
\centering
\begin{tabular}{cc}
\includegraphics[width=0.48\textwidth]{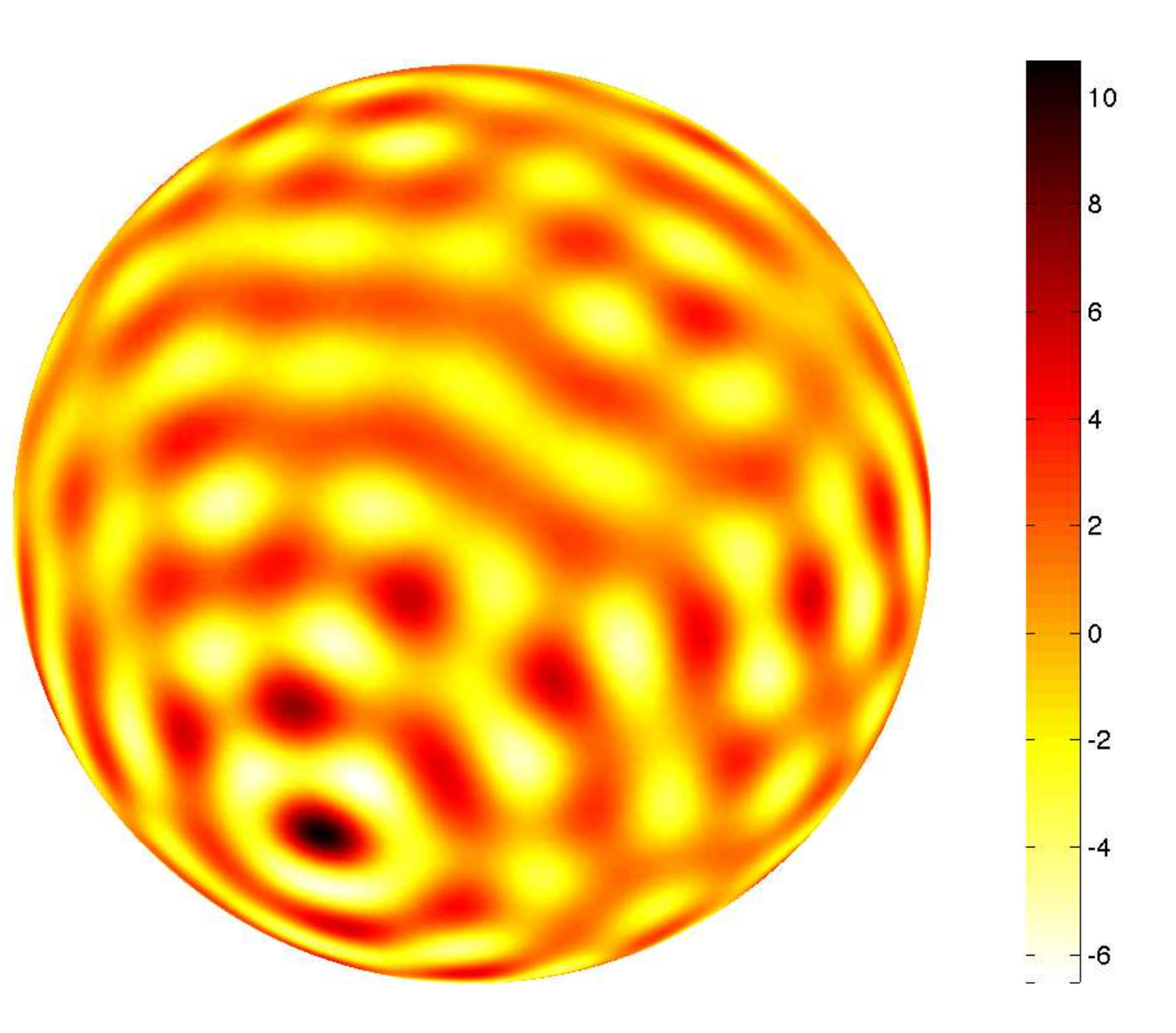} &
\includegraphics[width=0.48\textwidth]{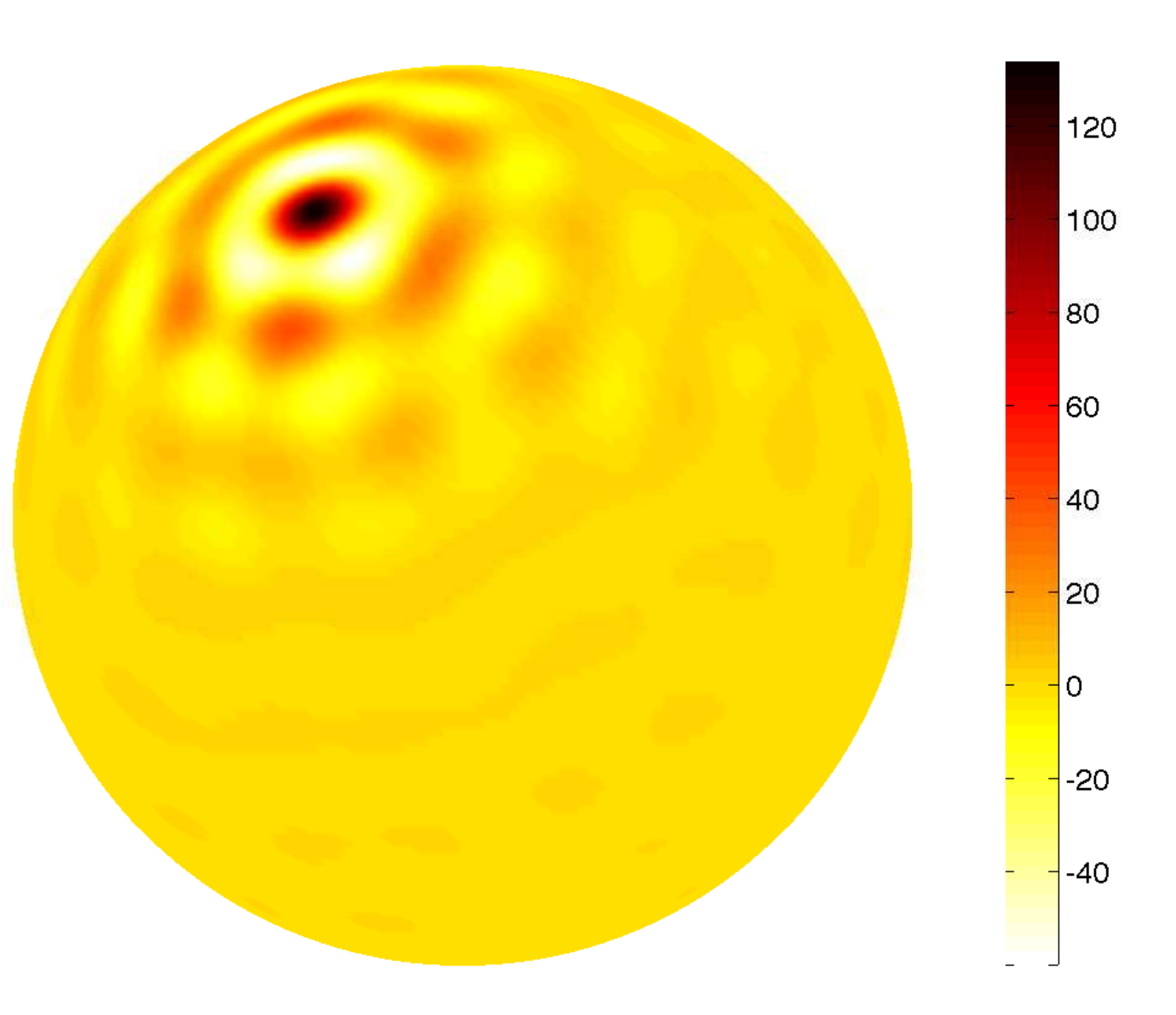} \\
(a) $f_1$ & (b) $f_2$ \\
\end{tabular}
\caption{Target integrands used to test the error estimates. (a) $f_1(x)$ from \eqref{eq:rough_target} and (b) $f_2(x)$ from \eqref{eq:smooth_target}.\label{fig:target_integrands}}
\end{figure}

Using the above results, we define the following two target integrands:
\begin{align}
f_1(x) &= \sum_{k=1}^{41} \text{sign}(Y_{20,k}(x_c))Y_{20,k}(x)g_1(x\cdot x_c), \label{eq:rough_target}\\
f_2(x) &= \sum_{k=1}^{41} \text{sign}(Y_{20,k}(x_c))Y_{20,k}(x)g_2(x\cdot x_c), \label{eq:smooth_target}
\end{align}
where $x_c=(\cos(-2.0281)\sin(0.76102),\sin(-2.0281)\cos(0.76102),\sin(0.76102))$ and $\varepsilon = 2/3$ for $g_2$; see Figure \ref{fig:target_integrands} (a) and (b) for plots of these respective functions.  Integrating these functions over $\sph^2$ and applying the Funk-Hecke formula \eqref{eq:funk_hecke} gives
\begin{align}
\int_{\sph^2} f_1(x) d\mu(x) &= 0.014830900415995, \\
\int_{\sph^2} f_2(x) d\mu(x) &= 0.032409262543520.
\end{align}
Both $f_1$ and $f_2$ inherit their smoothness directly from $g_1$ and $g_2$, respectively.  Thus, $f_1$ belongs to every Sobolev space $W_2^{\mu}(\sph^2)$ with $\mu < \frac52$ and $f_2\in C^{\infty}(\sph^2)$.

Tables \ref{tbl:rough_err} and \ref{tbl:smooth_err} display the relative errors in the $N$-point quadrature of $f_1$ and $f_2$, respectively, for the different family of nodes, while Figures \ref{fig:target_err}(a) and (b) display these respective results graphically on a $\log-\log$ scale.  Focusing first on the results for $f_1$ in Figure \ref{fig:target_err}(a) and comparing the results to the included dashed line, it is clear that for all three families the error is decreasing approximately like $\calo(N^{-1.25})$.  Since $h \sim N^{-1/2}$ for these families of nodes, this observed rate of decrease in the error is approximately $\calo(h^{2.5})$, which is precisely the rate predicted by the estimates in Theorem \ref{accuracy_stabiliy_polyharmonic} for functions with $f_1$'s smoothness.   Doing a similar comparison of the results for $f_2$ in Figure \ref{fig:target_err}(b), we see that the errors associated with the icosahedral nodes very clearly decrease like $\calo(N^{-2})$.  The results for the other two nodes are not as clear, but for larger values of $N$ the errors do appear to be decreasing approximately like $\calo(N^{-2})$.  Again, because of the relationship between $h$ and $N$ for these nodes, the errors for $f_2$ are thus decreasing approximately like $\calo(h^4)$.  This is the expected rate from Corollary \ref{TPS_quadrature} since $f_2$ is infinitely smooth and we have used the $m=2$ surface spline.

\begin{table}[h]
\centering
\begin{tabular}{||r|c||r|c||r|c||}
\hline
\hline
\multicolumn{2}{||c||}{Icosahedral} & \multicolumn{2}{c||}{Fibonacci} & \multicolumn{2}{c||}{Min. Energy} \\
N & Rel. Error & N & Rel. Error & N & Rel. Error \\
\hline
\hline
$2562$ & $1.926\times 10^{-1}$ & $2501$ & $5.112\times 10^{-3}$ & $2500$ & $3.048\times 10^{-2}$ \\
$10242$ & $3.533\times 10^{-2}$ & $10001$ & $5.549\times 10^{-3}$ & $10000$ & $6.848\times 10^{-2}$ \\
$23042$ & $1.286\times 10^{-2}$ & $22501$ & $1.770\times 10^{-3}$ & $22500$ & $2.480\times 10^{-2}$ \\
$40962$ & $6.268\times 10^{-3}$ & $40001$ & $1.040\times 10^{-3}$ & $40000$ & $1.217\times 10^{-2}$ \\
$92162$ & $2.273\times 10^{-3}$ & $62501$ & $6.460\times 10^{-4}$ & $62500$ & $6.989\times 10^{-3}$ \\
$163842$ & $1.107\times 10^{-3}$ & $90001$ & $4.068\times 10^{-4}$ & $90000$ & $4.393\times 10^{-3}$ \\
$256002$ & $6.330\times 10^{-4}$ & $160001$ & $1.844\times 10^{-4}$ & $160000$ & $2.083\times 10^{-3}$ \\
$655362$ & $1.957\times 10^{-4}$ & $250001$ & $1.085\times 10^{-4}$ & $250000$ & $1.228\times 10^{-3}$ \\  
\hline
\hline
\end{tabular}
\caption{Relative error in the $N$-point quadrature of the ``rough'' target function $f_1$ in \eqref{eq:rough_target} for the different families of nodes.\label{tbl:rough_err}}
\end{table}

\begin{table}[h]
\centering
\begin{tabular}{||r|c||r|c||r|c||}
\hline
\hline
\multicolumn{2}{||c||}{Icosahedral} & \multicolumn{2}{c||}{Fibonacci} & \multicolumn{2}{c||}{Min. Energy} \\
N & Rel. Error & N & Rel. Error & N & Rel. Error \\
\hline
\hline
$2562$ & $3.358\times 10^{-2}$ & $2501$ & $1.045\times 10^{-4}$ & $2500$ & $6.951\times 10^{-2}$ \\
$10242$ & $1.888\times 10^{-3}$ & $10001$ & $4.690\times 10^{-5}$ & $10000$ & $5.932\times 10^{-4}$ \\
$23042$ & $3.642\times 10^{-4}$ & $22501$ & $3.189\times 10^{-6}$ & $22500$ & $1.077\times 10^{-4}$ \\
$40962$ & $1.143\times 10^{-4}$ & $40001$ & $7.437\times 10^{-6}$ & $40000$ & $2.730\times 10^{-5}$ \\
$92162$ & $2.245\times 10^{-5}$ & $62501$ & $1.805\times 10^{-6}$ & $62500$ & $8.276\times 10^{-6}$ \\
$163842$ & $7.098\times 10^{-6}$ & $90001$ & $9.204\times 10^{-7}$ & $90000$ & $6.944\times 10^{-6}$ \\
$256002$ & $2.897\times 10^{-6}$ & $160001$ & $3.009\times 10^{-7}$ & $160000$ & $1.250\times 10^{-6}$ \\
$655362$ & $4.433\times 10^{-7}$ & $250001$ & $1.411\times 10^{-7}$ & $250000$ & $5.414\times 10^{-7}$  \\
\hline
\hline
\end{tabular}
\caption{Relative error in the $N$-point quadrature of the ``smooth'' target function $f_2$ in \eqref{eq:smooth_target} for the different families of nodes.\label{tbl:smooth_err}}
\end{table}

\begin{figure}[ht]
\centering
\begin{tabular}{cc}
\includegraphics[width=0.48\textwidth]{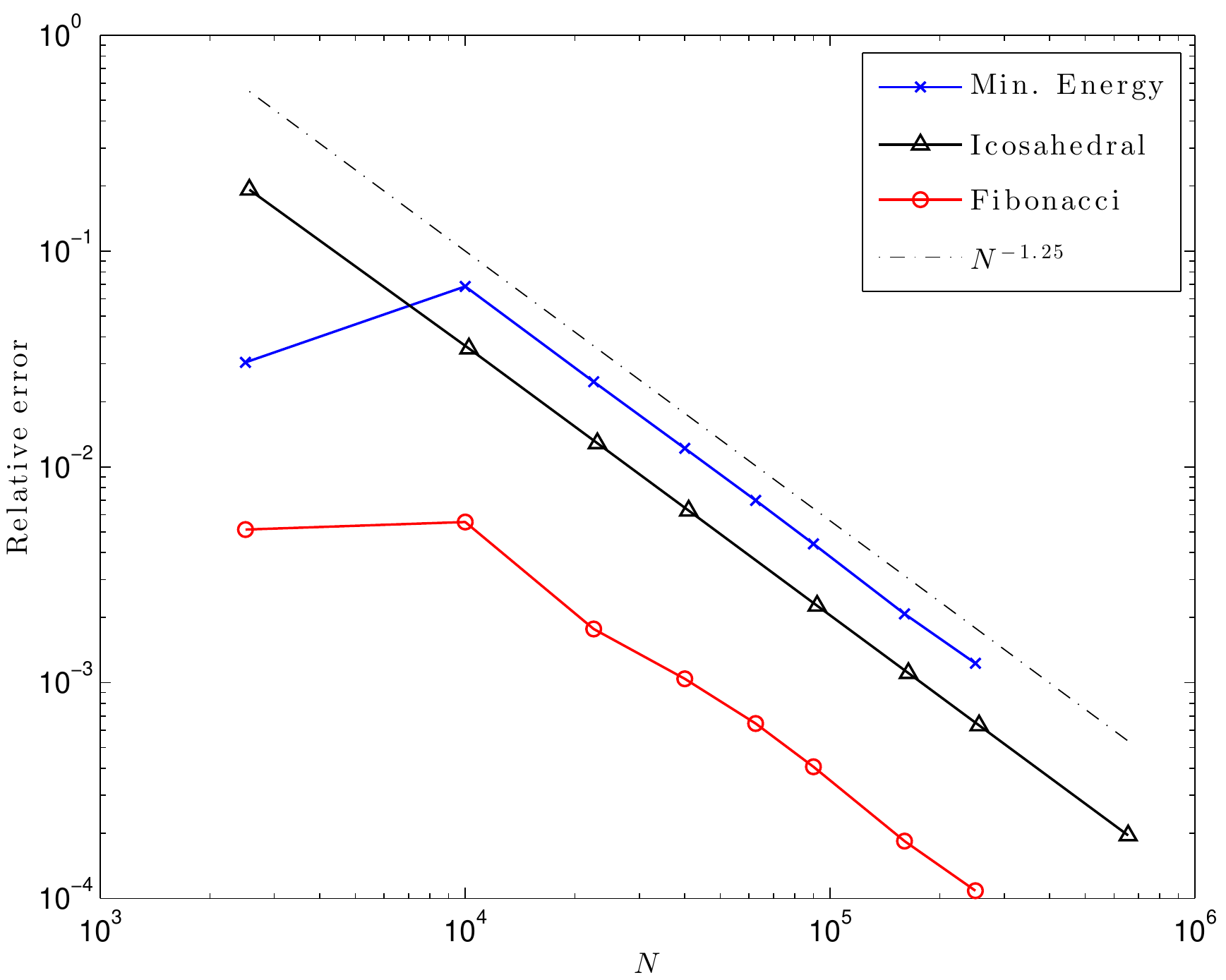} &
\includegraphics[width=0.48\textwidth]{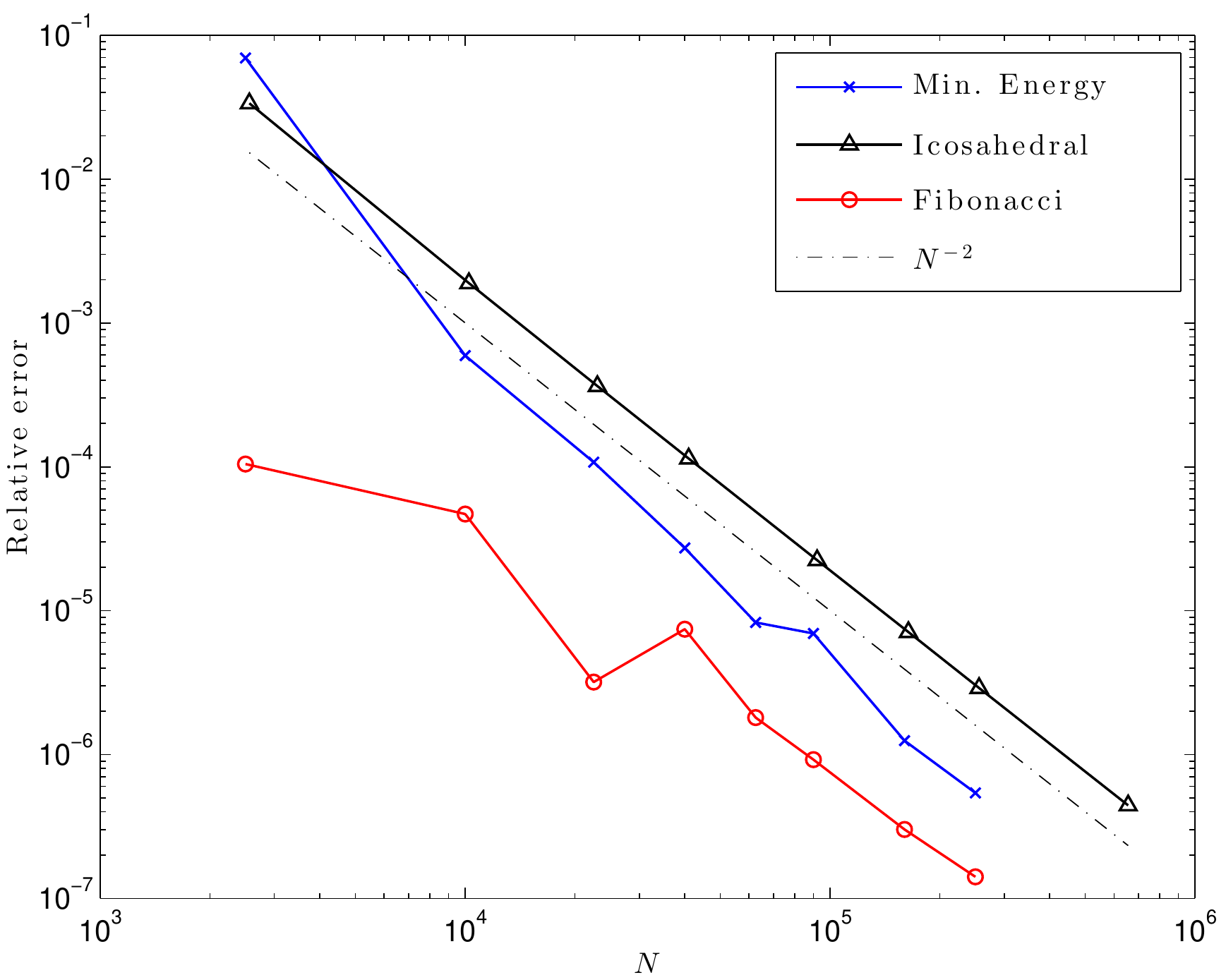} \\
(a) Rough target, $f_1$. & (b) Smooth target, $f_2$. \\
\end{tabular}
\caption{Relative errors in the $N$-point quadratures of (a) $f_1$ and (b) $f_2$ for the three different families of nodes.  The mesh-norm for all three families of nodes satisfies $h\sim N^{-1/2}$, so that the dashed lines in the figures indicate convergence like (a) $\calo(h^{2.5})$ and (b) $\calo(h^4)$.  These are the predicted convergence rates from Theorem \ref{accuracy_stabiliy_polyharmonic} and Corollary \ref{TPS_quadrature}, respectively. \label{fig:target_err}}
\end{figure}

We conclude by noting that the nodes and quadrature weights used in the numerical experiments above are available for download from~\cite{WrightQuadWeights}.

\section{Quadrature on Manifolds Diffeomorphic to Homogeneous Spaces}
\label{manif_diffeo}
Numerical integration of functions defined on smooth surfaces that are diffeomorphic to $\sph^2$ arise in a number of applications.  For example, the shape of the earth, as well as the other planets, is a ``flattened sphere'' -- i.e., an  oblate spheroid, which is diffeomorphic to a sphere. To numerically compute integrals over the earth's surface then requires quadrature formulas over oblate spheroids. In addition, the need for numerical integration over various surfaces also arises in boundary element formulations of continuum problems in $\RR^3$~\cite[Ch. 6]{atkinson2012spherical}.  In this section, we discuss quadrature in a more general context. We will show how to use quadrature weights for a  homogeneous manifold $\M$ to obtain an invariant, coordinate independent quadrature formula for a smooth Riemannian manifold $\LL$ diffeomorphic to $\M$.  Of course,  the case in which $\LL$ is an oblate spheroid and $\M=\sph^2$ is of special interest.

Recall that a $C^\infty$ manifold $\LL$ is diffeomorphic to $\M$ if there is a $C^\infty$ bijection $F:\M \to \LL$. Let $g_{ij}$ be the Riemannian metric for $\M$. Suppose that $\LL$ has the metric $h_{ij}$. Consider a local chart $(U,\phi:U\to \RR^n)$ near $p_0\in \M$. Using this chart we have coordinates  $x=(x^1,\ldots,x^n) = \phi(p)$ and a parametrization $p=\phi^{-1}(x)$. 

We can use $(U,\phi)$ to produce a local chart $(V,\psi:V \to \RR^n)$ near $q_0=F(p_0) \in \LL$.  Simply let $V=F(U) \subset \LL$ and $\psi = \phi\circ F^{-1}$. The coordinates for $V$ are thus $x = \psi(q)$, and so the two metrics $g_{ij}$ and $h_{ij}$ can be expressed in terms of the same set of coordinates. In these coordinates, the volume elements are given by 
\[
d\mu_{\M}(x) = \sqrt{\det(g_{ij}(x))}\,dx^1\cdots dx^n \ \text{and} \  d\mu_\LL(x) = \sqrt{\det(h_{ij}(x)}\,dx^1\cdots dx^n.
\]
It follows that
\begin{equation}
\label{volume_relation}
d\mu_\LL(x) = \underbrace{\sqrt{\frac{\det(h_{ij}(x))}{\det(g_{ij}(x)) }}}_{\displaystyle{w(x)}} d\mu_{\M}(x).
\end{equation}
Suppose that we now make a change of coordinates, from $x=\phi(p)$ to new coordinates $y=\varphi(p)$, or $x=x(y)$. Let $x'(y)$ be the Jacobian matrix for the transformation, and let $J(y)=\det(x'(y))$. In $y$ coordinates, the metrics are $\tilde g(y)=(x'(y))^Tg(x(y))x'(y)$ and $\tilde h(y)=(x'(y))^T h(x(y))x'(y)$. Consequently, we have that 
\[
\det(\tilde g_{ij}(y)) = J(y)^2 \det(g_{ij}(x(y))) \ \text{and} \ \det(\tilde h_{ij}(y)) = J(y)^2 \det(h_{ij}(x(y))),
\]
and, furthermore, that
\[
w(x(y))=\sqrt{\frac{\det(h_{ij}(x(y)))}{\det(g_{ij}(x(y))) }} = \sqrt{\frac{\det(h'_{ij}(y))}{\det(g'_{ij}(y)) }} = \tilde w(y).
\]
This means that $w\circ \phi(p) = \tilde w\circ \psi(p)=:W(p)$ is thus a scalar invariant that is independent of the choice of coordinates. In terms of integrals, we have
\begin{equation}
\label{intLL_intM}
\int_{\LL} f(q) d\mu_\LL(q) = \int_{\M} f\circ F(p) W(p) d\mu_{\M}(p).
\end{equation}

An invariant, coordinate independent  quadrature formula for $\LL$ can be obtained from the one for $\M$. We have the following result.

\begin{proposition}\label{quad_diffeo}
Let $X$ denote the set of centers on $\LL$ and let $X'=F^{-1}(X)$ be the corresponding set on $\M$. Suppose that we have a quadrature formula for $\M$ with weights $\{C_{\xi'}\}_{\xi'\in X'}$. Then we have the following quadrature formula for $\LL$,
\begin{equation}\label{quad_diffeo_formula}
Q_\LL(f) := \sum_{\xi \in X} f(\xi)c_\xi,\ \text{where }c_\xi := W(F^{-1}(\xi))C_{F^{-1}(\xi)}.
\end{equation}
\end{proposition}

\begin{proof}
Applying the quadrature formula for the homogeneous manifold $\M$ to the integral on the right-hand side of (\ref{intLL_intM}) yields
\[
Q_{\M}\big(f\circ F(p) W(p)\big) = \sum_{\xi'\in X'} f\circ F(\xi') W(\xi')C_{\xi'}.
\]
Since $\xi' = F^{-1}(\xi)$, we have $f\circ F(\xi') = f(\xi)$ and $W(\xi')=W(F^{-1}(\xi))$.  Taking $c_\xi =  W(F^{-1}(\xi))C_{F^{-1}(\xi)}$ we obtain (\ref{quad_diffeo_formula}).
\end{proof}

\paragraph{Oblate spheroid} Consider the oblate spheroid $\LL$, $x^2+y^2+z^2/a^2=1$, where $0<a <1$, and the 2-sphere $\sph^2(=\M)$, $X^2+Y^2+Z^2=1$. The diffeormorphism between the two manifolds is $F(X,Y,Z) := (X,Y,aZ)$; that is, $(x,y,z) = (X,Y,aZ)$. Our aim is to find the scale factor $W$. Since the end result will be coordinate independent, we  choose to work in spherical coordinates  $(\theta,\phi)$, where $\theta$ is the longitude and $\phi$ is the latitude on $\sph^2$. (The north pole is $(0,0,1)$.) Obviously, for $\LL$ we have $x = \sin \theta \cos \phi$, $y=\sin \theta \sin \phi$, $z = a \cos\theta$. The metric for the sphere is $dS^2 = d\theta^2 + \sin^2(\theta)d\phi^2$. The metric for $\LL$ is the Euclidean metric $dx^2+dy^2+dz^2 $ on $\RR^3$ restricted to $\LL$. Making a straightforward calculation, one can show that the metric for $\LL$ is
\[
ds^2 = (\cos^2 \theta +a^2\sin^2\theta )d\theta ^2 +\sin^2 \theta d\phi^2 = \big(a^2 + (1-a^2)\cos^2 \theta\big) d\theta^2 +\sin^2 \theta d\phi^2. 
\]
Consequently, the volume element for $\LL$ is 
\[
d\mu_\LL = \sqrt{a^2 + (1-a^2)\cos^2 \theta\,}\sin \theta \,d\theta d\phi = \underbrace{\sqrt{a^2 + (1-a^2)\cos^2 \theta\,}}_{w(\theta,\phi)}\,d\mu_{\sph^2}
\] 
To put this in invariant form on $\sph^2$, we use $Z=\cos\theta$ to obtain $W(X,Y,Z)=\sqrt{a^2 + (1-a^2)Z^2\,}$. Pulling this back to $\LL$, we have $W\circ F^{-1}(x,y,z)= \sqrt{a^2 + (a^{-2}-1)z^2\,}$. The weights for $Q_\LL$ are thus
\[
c_\xi = \sqrt{a^2 + (a^{-2}-1)\xi_z^2\,}\,C_{(\xi_x,\xi_y, \xi_z/a)}.
\]

As we mentioned above, the earth is approximately an oblate spheroid. The parameter $a$ is the ratio of the polar radius to the equatorial  radius. The \emph{flattening} of the earth is $f = 1 -a$, and has the approximate value $f\approx 1/300$ (cf.\ Earth Fact Sheet \cite{NASA-fact-sheet-2008}). From this, we get that $a\approx 299/300$, and so  $W\circ F^{-1}(x,y,z) \approx \sqrt{0.993+ 0.007z^2}$. This factor varies between $0.9967$ and $1$, about $0.3\%$. Even so, it could affect the accuracy of the quadrature formula for functions with large values near the equator. As an aside, Jupiter has $f\approx 1/15$ (cf.\ Jupiter Fact Sheet (ellipticity) \cite{NASA-fact-sheet-2008}), and for it the change in $W\circ F^{-1}$ would be a hefty $7\%$.

\section*{Acknowledgments}
We thank Prof. Doug Hardin from Vanderbilt University for providing us with code for generating the quasi-minimum energy points used in the numerical examples based on the technique described in~\cite{BorodachovHardinSaff:2012}.  

\bibliographystyle{siam} \bibliography{rbf}

\end{document}